%% file: limited_memory_comparison.tex
\documentclass{siamltex}

\usepackage[nosumlimits]{amsmath}
\usepackage{amsfonts,amssymb}
\usepackage{graphicx}
\graphicspath{{./Figures/}}
\usepackage{color}
\usepackage{ulem}
\usepackage{multicol}
\usepackage{algorithm2e}
\usepackage{cite} 
\usepackage{placeins}
\usepackage{pifont}
\usepackage[dvipsnames]{xcolor}
\usepackage{tikz}
\usepackage{pgfplots}
\usetikzlibrary{shapes}

\newcommand{\cmark}{\ding{51}}%
\newcommand{\xmark}{\ding{55}}%
\DeclareMathAlphabet{\mathbf}{OT1}{cmr}{bx}{it}

\newcommand{\vb}{{\mathbf b}}

\newcommand{\ve}{{\mathbf e}}
\newcommand{\vehat}{{\mathbf {\hat e}}}
\newcommand{\vf}{{\mathbf f}}

\newcommand{\vfMS}{\vf^{MS}}
\newcommand{\vg}{{\mathbf g}}
\newcommand{\vgSI}{{\mathbf g}^{SI}}

\newcommand{\vghatSI}{{\mathbf {\hat g}^{SI}}}

\newcommand{\vr}{{\mathbf r}}

\newcommand{\vv}{{\mathbf v}}

\newcommand{\vvSI}{{\mathbf v}^{SI}}
\newcommand{\vw}{{\mathbf w}}
\newcommand{\vx}{{\mathbf x}}
\newcommand{\vy}{{\mathbf y}}

\newcommand{\vnull}{\boldsymbol{0}}
\newcommand{\spK}{{\cal K}}
\newcommand{\spEK}{{\cal EK}}

\newcommand{\lmin}{\lambda_{\min}}
\newcommand{\lmax}{\lambda_{\max}}

\newcommand{\dmu}{\d\mu(t)}

\renewcommand{\d}{\,\mathrm{d}}
\newcommand{\R}{\mathbb{R}}

\newcommand{\Rmm}{\mathbb{R}^{m \times m}}
\newcommand{\C}{\mathbb{C}}

\newcommand{\CN}{\mathbb{C}^N}

\newcommand{\CNN}{\mathbb{C}^{N \times N}}

\newcommand{\CNm}{\mathbb{C}^{N \times m}}

\newcommand{\VSI}{V^{SI}}
\newcommand{\TSI}{T^{SI}}

\newcommand{\W}{\mathcal{W}}
\newcommand{\V}{\mathcal{V}}
\newcommand{\M}{\mathcal{M}}
\newcommand{\E}{\mathcal{E}}
\newcommand{\BS}{\mathcal{S}}

\DeclareMathOperator{\Span}{span}
\DeclareMathOperator{\spec}{spec}

\newtheorem{example}[theorem]{Example}
\let\oldexample\example
\renewcommand{\example}{\oldexample\normalfont}
\newtheorem{remarksimple}[theorem]{Remark}
\let\oldremarksimple\remarksimple
\renewcommand{\remarksimple}{\oldremarksimple\normalfont}
\newenvironment{remark}{\begin{remarksimple}}{\hfill$\diamond$\end{remarksimple}}
\newtheorem{experiment}[theorem]{Experiment}
\let\oldexperiment\experiment
\renewcommand{\experiment}{\oldexperiment\normalfont}

\newcommand{\review}[1]{{#1}}
\title{A comparison of limited-memory Krylov methods\\ for Stieltjes functions of Hermitian matrices}
\author{Stefan G\"uttel\thanks{Department of Mathematics, The University of Manchester, M13\,9PL Manchester, United Kingdom, \texttt{stefan.guettel@manchester.ac.uk}.  S.\,G. acknowledges financial support by The Alan Turing Institute under the EPSRC grant EP/N510129/1.} \and Marcel Schweitzer\thanks{Mathematisch-Naturwissenschaftliche Fakult\"at, Heinrich-Heine-Universit\"at D\"usseldorf, Universit\"atsstra\ss{}e 1, 40225 D\"usseldorf, Germany, \texttt{marcel.schweitzer@hhu.de}}}
\date{\today}
\begin{document}
\renewcommand{\thefootnote}{\fnsymbol{footnote}}
\maketitle \pagestyle{myheadings} \thispagestyle{plain}
\markboth{S. G\"UTTEL AND M. SCHWEITZER}{LIMITED-MEMORY KRYLOV FOR STIELTJES FUNCTIONS}

\begin{abstract}
Given a limited amount of memory and a target accuracy, we propose and compare several polynomial Krylov methods for the approximation of $f(A)\vb$, the action of a Stieltjes matrix function of a large Hermitian matrix on a vector. Using new error bounds and estimates, as well as existing results, 
we derive predictions of the practical performance of the methods, and rank them accordingly. As by-products, we derive new results on inexact Krylov iterations  for matrix functions in order to allow for a fair comparison of rational Krylov methods with polynomial inner solves.

\end{abstract}




\section{Introduction}

In recent years considerable progress has been made in the development of numerical methods for the efficient approximation of $f(A)\vb$, 
the action of a matrix function $f(A)$ on a vector $\vb$. In applications, the matrix $A\in\CNN$ is typically  large and sparse and the computation of the generally dense matrix $f(A)$ is infeasible. One therefore seeks to approximate $f(A)\vb$ directly by means of some iterative method. By far the most popular  methods for this task are polynomial~\cite{DruskinKnizhnerman1989,Saad1992} or rational~\cite{Guettel2013,GuettelKnizhnerman2013,vdEH06,DruskinKnizhnerman1998} Krylov methods. \review{In the latter class of methods is, in particular, the popular extended Krylov subspace method \cite{DruskinKnizhnerman1998,KnizhnermanSimoncini2010} which utilizes matrix-vector products and linear system solves with the matrix~$A$. In cases where the matrix size  is such that direct solution methods for (shifted) linear systems with $A$ are feasible, or in cases where a good preconditioner is available to solve such problems iteratively, rational Krylov methods can significantly outperform polynomial methods. On the other hand, even if applicable, rational Krylov methods can be somewhat more difficult to tune as generally more parameters need to be chosen to obtain fast convergence. There are also variants of rational Krylov methods for $f(A)\mathbf{b}$ that choose their shift parameters automatically based on some heuristics (see, e.g.,  \cite{druskin2010adaptive,GuettelKnizhnerman2013}), but then there is very little theory that governs their convergence.} 

In this work we investigate which polynomial Krylov methods are best suited \emph{when $A$ is Hermitian and the only feasible operation involving this matrix are matrix-vector products.} This might be the case, e.g., when direct solvers are inefficient due to $A$'s sparsity structure or if $A$ is only implicitly available through a routine that returns the result of a matrix-vector product. We further assume that memory is limited so that only a predefined  number $m_{\max}$ of vectors of size $N$ can be stored. This situation arises in many applications, including lattice quantum chromodynamics~\cite{BlochEtAl2007,VanDenEshofFrommerLippertSchillingVanDerVorst2002,BrannickFrommerKahlLederRottmannStrebel2014} and statistical sampling~\cite{IlicTurnerPettitt2004,IlicTurnerSimpson2010,SimpsonTurnerPettitt2008}.

The basis of polynomial Krylov methods for Hermitian matrices is the \emph{Lanczos method}~\cite{Lanczos1950} (which corresponds to the \emph{Arnoldi method}~\cite{Arnoldi1951} in the non-Hermitian case). Applying $m$ Lanczos iterations with $A$ and $\vb$ yields the \emph{Lanczos relation}
\begin{equation}\label{eq:lanczos_relation}
AV_m = V_mT_m + \beta_{m+1}\vv_{m+1}\ve_m^T,
\end{equation}
with 
$
V_m = \left[\vv_1, \vv_2, \ldots, \vv_m \right] \in \CNm
$
containing an orthonormal basis of the Krylov space 
$
\spK_m(A,\vb) := \Span\{\vb, A\vb, A^2\vb, \dots, A^{m-1}\vb\},
$
a symmetric tridiagonal matrix
$$
T_m = \left[\begin{array}{ccccc} 
\eta_1 & \beta_2 & & & \\
\beta_2 & \eta_2 & \beta_3 & & \\
 & \ddots & \ddots & \ddots & \\
 & & \beta_{m-1} & \eta_{m-1} & \beta_m \\
 & & & \beta_m & \eta_m
\end{array}\right] \in \Rmm, 
$$
and $\ve_m$ denoting the $m$th canonical unit vector in $\mathbb{R}^m$. 
The \emph{Lanczos approximation} $\vf_m \approx f(A)\vb$ is obtained by projecting the original problem onto the Krylov space, 
\begin{equation}\label{eq:lanczos_approximation}
\vf_m := V_mf(V_m^HAV_m)V_m^H\vb = \|\vb\|V_mf(T_m)\ve_1,
\end{equation}
where $\|\cdot\|$ denotes the Euclidean vector norm. The evaluation of~\eqref{eq:lanczos_approximation} requires the storage of the full Lanczos basis $V_m$, i.e., $m$ vectors of length~$N$. In the situation described above, the value $m_{\max}$ which limits the number of vectors that can be stored therefore limits the maximum number of iterations that can be performed and thus also the attainable accuracy of the Lanczos approximation. This is different from the situation for Hermitian linear systems, where the short recurrence for the Krylov basis vectors translates into a short recurrence for the iterates, resulting in the famous conjugate gradient method~\cite{HestenesStiefel1952}.

There are several approaches for overcoming the memory problem, including
\begin{itemize}
\item the two-pass Lanczos method, see, e.g.,~\cite{Borici2000,FrommerSimoncini2008} which overcomes memory limitations but roughly doubles the computational effort,
\item the multi-shift CG method~\cite{VanDenEshofFrommerLippertSchillingVanDerVorst2002,FrommerMaass1999}, which replaces $f$ by a rational approximation and then simultaneously solves the resulting linear systems using the short recurrence conjugate gradient method,
\item restarted Krylov methods~\cite{AfanasjewEtAl2008a,EiermannErnst2006,FrommerGuettelSchweitzer2014a,FrommerGuettelSchweitzer2014b,IlicTurnerSimpson2010,Schweitzer2016thesis,TalEzer2007} which, similar to restarted methods for linear systems, construct a series of Krylov iterates in such a way that each ``cycle'' of the method only requires a fixed amount of storage.
\end{itemize}
We also refer the reader to the recent survey~\cite{GuettelKressnerLund2020} covering limited-memory polynomial methods for the $f(A)\vb$ problem.

Another approach, which has not been considered in this context in the literature so far, is to use rational Krylov methods~\cite{BerljafaGuettel2015, Guettel2013,GuettelKnizhnerman2013} combined with an iterative short recurrence solver for the associated linear systems. It appears to be generally thought that using a polynomial Krylov solver inside a rational Krylov method is not sensible, because the approximation computed is then polynomial and hence could also be computed with a polynomial Krylov method alone. While this is true in theory, an outer-inner rational-polynomial Krylov method may still be interesting in our setting: if the overall number of outer iterations is small, the number of vectors to be stored is small (as the inner iteration uses a short recurrence) and hence the method could be a viable alternative to restarting strategies. Thus, we pose the following question: 

\smallskip

\noindent\textbf{Given a limited amount of memory (storage of at most $m_{\max}$ vectors of length~$N$) and a target accuracy $\varepsilon$, what is an efficient way to extract an accurate approximation to  $f(A)\vb$ from a polynomial Krylov space?}

\smallskip

Of course, ``efficient'' can have several meanings like, e.g., ``small number of matrix-vector products and inner products'' or ``low overall computation time''. The latter criterion is highly dependent on the specific implementation of each method and also the hardware environment (e.g., parallel/distributed computing), and hence difficult to asses given only information on $f,A,\mathbf{b}$; see also  \cite{carson2020cost} for a discussion of such difficulties. 
We take a more general, implementation--independent approach here by exploiting the considerable progress that has recently been made in the understanding of (restarted) Krylov methods for the $f(A)\vb$ problem~\cite{FrommerGuettelSchweitzer2014a,FrommerGuettelSchweitzer2014b,Schweitzer2016thesis}. Together with the very well-understood convergence behaviour of the conjugate gradient method (see, e.g.,~\cite{Saad2003}), this opens up the possibility to assess and compare the efficiency of different algorithmic variants using upper error bounds.
We hope that this theoretical work will serve as a starting point for a more practice-oriented comparison of the different algorithms for
\begin{itemize}
\item[a)] investigating the potential for efficient parallelization and tuning,
\item[b)] comparing our theoretical estimates of iteration numbers to the real numbers occurring when solving different real-world problems.
\end{itemize}

Most of the available theoretical results mentioned above apply to the class of \emph{Stieltjes functions}
 \begin{equation}\label{eq:stieltjes_function}
f(z) = \int_0^\infty \frac{1}{t+z}\dmu,
\end{equation}
where \review{$\mu(t)$ is a monotonically increasing and nonnegative measure on $[0,\infty)$,  and 
$
\int_0^\infty (t+1)^{-1} \dmu < \infty
$; }
see, e.g.,~\cite{BergForst1975, Berg2007,Henrici1977}. \review{The latter condition ensures that $f(z)$ is finite for all $z>0$.} Important examples of Stieltjes functions include $f(z) = z^{-\alpha}$ with $\alpha \in (0, 1)$ and $f(z) = \log(1+z)/z$. Stieltjes matrix functions are closely related to shifted linear systems (see, e.g.,~\cite{FrommerGuettelSchweitzer2014b}), which allows us to transfer many theoretical results well-known for linear systems, like the classical CG convergence bound. In particular, the following theorem is central to many developments in this paper. 
The $A$-norm used in the statement of the theorem is defined as $\|\mathbf x\|_A := \sqrt{\mathbf x^H A \mathbf x}$.

\begin{theorem}[see, e.g.,~\cite{Saad2003}]\label{the:cg_convergence}
Let $A \in \CNN$ be Hermitian positive definite and $\vx_0,\vb \in \CN$. Further, let $\vx^\ast$ denote the solution of the linear system $A\vx = \vb$ and let $\vx_m$ be the $m$th CG iterate with initial guess $\vx_0$. Let $\lmin$ and $\lmax$ denote the smallest and largest eigenvalue of $A$, respectively, and let $\kappa = \frac{\lmax}{\lmin}$ denote the condition number of $A$. Define
\begin{equation*}\label{eq:c_alpha_cg}
c = \frac{\sqrt{\kappa} - 1}{\sqrt{\kappa} + 1} \enspace \text{and} \enspace \alpha_m = \frac{1}{\cosh(m\ln c)}
\end{equation*}
(where we set $\alpha_m = 0$ if $\kappa = 1$). Then the error in the CG method satisfies
\begin{equation*}
\|\vx^\ast - \vx_m\|_A \leq \alpha_m \|\vx^\ast-\vx_0\|_A.
\end{equation*}
\end{theorem}

The remainder of this paper is structured as follows. In Section~\ref{sec:polynomial_methods} we give a survey of the different established polynomial methods for approximating $f(A)\vb$, together with their (worst-case) convergence bounds. Sections~\ref{sec:rational_methods} and \ref{sec:inexact_iterations}  introduce new combinations of outer-inner rational-polynomial Krylov methods, namely an inexact shift-and-invert method~\cite{vdEH06,Moret2009,MoretPopolizio2014} and an inexact extended Krylov method~\cite{DruskinKnizhnerman1998,KnizhnermanSimoncini2010} with polynomial Krylov solvers. We provide a new convergence analysis for the shift-and-invert method and discuss ways to relax the inner iterations of the proposed methods. In Section~\ref{sec:comparison} we use the obtained convergence results to estimate the  total arithmetic cost of each of the considered methods, and discuss  general advantages, disadvantages, and prerequisites of each of the methods. The theoretical estimates are compared to real iteration counts for some (artificial) test problems. Concluding remarks and topics for future research are given in Section~\ref{sec:conclusion}.

\section{Polynomial limited-memory Krylov methods}\label{sec:polynomial_methods}
As the problem of approximating functions of very large, sparse matrices arises frequently in applications, several different strategies have been developed to overcome the problem of scarce memory. We briefly describe three established Krylov methods, together with theoretical results on their convergence behaviour.

\subsection{Two-pass Lanczos}\label{subsec:two_pass_lanczos}

The two-pass Lanczos method~\cite{Borici2000,FrommerSimoncini2008} is a very simple approach that solves the scarce memory problem by applying the Lanczos process twice. Of course, this doubles the number of matrix-vector products and inner products that need to be computed.

In the first pass of the Lanczos method one computes the compressed matrix~$T_m$ and discards the basis vectors in $V_m$ as soon as they are no longer needed to compute the next basis vector (i.e., only storing the last three basis vectors). Then, the coefficient vector
$\vy_m := \|\vb\|f(T_m)\ve_1
$
is computed. In the second pass, the Lanczos approximation is computed as
$$
\vf_m = \sum\limits_{i=1}^m [\vy_m]_i \vv_i,
$$
which can be updated from one iteration to the next and thus also allows to discard old basis vectors. 
This approach produces the same iterates as the standard Lanczos process but requires twice the number of matrix-vector products.  When $f$ is a Stieltjes function, the convergence of the two-pass Lanczos method is thus characterized by the following theorem from~\cite{FrommerGuettelSchweitzer2014b}. It is a special case of a more general result for the restarted Lanczos method; cf.~Theorem~\ref{the:restarted_lanczos_convergence_stieltjes} below.
\begin{theorem}[Corollary 4.4 in \cite{FrommerGuettelSchweitzer2014b}]\label{the:lanczos_convergence_stieltjes}
Let $A \in \CNN$ be Hermitian positive definite, $\vb \in \CN$, $f$ a Stieltjes function~\eqref{eq:stieltjes_function}, and let $\vf_m$ be the approximation to $f(A)\vb$ after $m$ iterations of the Lanczos method. Further, let 
\begin{equation}\label{eq:aux_functions}
\kappa(t) = \frac{\lmax+t}{\lmin+t}, \, c(t) = \frac{\sqrt{\kappa(t)}-1}{\sqrt{\kappa(t)}+1}, \, \text{ and }\alpha_m(t) = \frac{1}{\cosh(m \ln c(t))}.
\end{equation}
and let $t_0 \geq 0$ be the left endpoint of the support of $\mu$. Then
\begin{equation}\label{eq:errorbound_unrestarted}
\|f(A)\vb - \vf_m\|_A \leq C \alpha_m(t_0),
\end{equation}
where
\begin{equation}\label{eq:constant_c}
\textstyle C = \|\vb\| \sqrt{\lmax} \cdot  f\big(\sqrt{\lmin\lmax}\big)
\end{equation}
\end{theorem}

We remark that the bound~\eqref{eq:errorbound_unrestarted} is, up to the factor~$C$, the same as the standard textbook CG convergence bound for the linear system $(A + t_0 I)\vx = \vb$.
Also, note that if additional computational work is invested into the computation of error bounds or error estimates during the first pass of the Lanczos method, then this work can of course be avoided in the second pass, as the number of iterations required to reach the desired accuracy is already known.

\review{%
\begin{remark}\label{rem:druskin}
In a massively parallel setting, an alternative to using a two-pass Lanczos approach is to compute $f(A)\vb$ element-wise, preferably by using the relation
\begin{equation}\label{eq:fAb_elementwise}
[f(A)\vb]_i = \ve_i^H\!f(A)\vb = \frac14 (\vb+\ve_i)^H\!f(A)(\vb+\ve_i) - \frac14 (\vb-\ve_i)^H\!f(A)(\vb-\ve_i).
\end{equation}
Thus, the full vector $f(A)\vb$ can be formed by evaluating $2N$ bilinear forms $\vv^H\!f(A)\vv$. If at least $2N$ processors are available---and if energy consumption is not an issue---one can perfectly parallelize the computation of $f(A)\vb$ in this way. This approach then has several upsides: Most importantly, when approximating bilinear forms $\vv^H\!f(A)\vv$ by the Lanczos process, it is not necessary to store the Lanczos basis. Additionally, due to the relation to Gaussian quadrature it can be expected that the necessary number of iterations is about half the number of iterations that would be required for approximating $f(A)\vv$ and moreover, this approach is less prone to instability due to rounding errors; see, e.g.,~\cite{GolubMeurant2010}. Thus, in the mentioned setting, an approach based on~\eqref{eq:fAb_elementwise} can be expected to outperform the other methods that we discuss in this paper.
\end{remark}}

\subsection{Multi-shift CG}\label{subsec:multi-shift_cg}

The multi-shift CG method~\cite{FrommerMaass1999,FrommerSimoncini2008,VanDenEshofFrommerLippertSchillingVanDerVorst2002} for $f(A)\vb$ is based on first approximating $f$ by a suitable rational function $r$, $r(A)\vb \approx f(A)\vb$, most commonly in partial fraction form
\begin{equation}\label{eq:pfe1}
r(z) = \sum\limits_{i=1}^p \omega_i\frac{1}{t-\zeta_i} \quad\text{ or }\quad r(z) = \sum\limits_{i=1}^p \omega_i\frac{t}{t^2-\zeta_i}, 
\end{equation}
where we assume that all poles $\zeta_i$ lie on the negative real axis.
Then  
\begin{equation}\label{eq:fAb_pfe}
f(A)\vb \approx \sum\limits_{i=1}^p \omega_i(A-\zeta_iI)^{-1}\vb \quad\text{ or }\quad f(A)\vb \approx A\sum\limits_{i=1}^p \omega_i (A^2-\zeta_iI)^{-1}\vb, 
\end{equation}
respectively, which approximate $f(A)\vb$ as the weighted sum of solutions of shifted linear systems with $A$ or $A^2$. Rational functions of the former form in~\eqref{eq:pfe1} arise as \review{$[p-1,p]$-type Pad\'e approximants of Stieltjes functions or the matrix exponential (see, e.g., \cite{GN16}), while functions of the latter form naturally originate in the Zolotarev approximation of the sign function (see, e.g., \cite{VanDenEshofFrommerLippertSchillingVanDerVorst2002}).}

The main trick for efficiently evaluating~\eqref{eq:fAb_pfe} is the \emph{shift-invariance} of Krylov spaces,
that is 
$
\spK_m(A,\vb) = \spK_m(A + \zeta I,\vb)
$
for all $\zeta\in\C$. 
Therefore, when started with an initial guess $\vx_0(\zeta) = \vnull$ for all shifted systems, the Krylov spaces from which the conjugate gradient method extracts its approximants for the different systems all coincide. This can be exploited by simultaneously solving all shifted systems in the same iterative process, requiring the same number of matrix-vector products and inner products as the solution of a single system. 
The multi-shift CG implementation proposed in~\cite{FrommerMaass1999} results in the following computational and memory overhead compared to the standard Lanczos method (ignoring negligible scalar operations):
\begin{itemize}
\item[(i)] Two vectors of length $N$ need to be stored for each pole of the rational approximation, i.e.,  $2p$ additional vectors overall.
\item[(ii)] In each iteration, two vector additions and three vector scalings are performed (equalling to about two and a half inner products in cost) for each pole of the rational approximation, totalling to $2.5 m p$ additional inner products.
\item[(iii)] After performing the multi-shift CG method, the iterates of the individual systems need to be combined according to~\eqref{eq:fAb_pfe}. This results in an additional number of $p$ vector scalings and $(p-1)$ vector additions (and, in case of a rational approximation of the later form in~\eqref{eq:pfe1}, one additional matrix-vector product), which equals to about $p$ additional inner products.
\end{itemize}

\begin{remark}\label{rem:reduce_memory_mscg}
In~\cite{VanDenEshofFrommerLippertSchillingVanDerVorst2002} it is proposed to not store the iterates of the individual systems and then combine them at the end, but instead to directly combine them in each iteration to update the approximation to $f(A)\vb$. This way, only one additional vector needs to be stored per system, but the computational effort of roughly $2p$ vector additions and scalings is required in each iteration of the multi-shift CG method.
\end{remark}

\smallskip

The (worst-case) speed of convergence of the multi-shift CG method is also described by the classical textbook CG convergence bound, with the worst-conditioned system (i.e., the one corresponding to the pole $\zeta_i$ closest to zero) determining the overall necessary number of iterations. When one of the shifts is very close to zero, as it is typically the case for rational approximations of Stieltjes functions, the resulting convergence factor is thus approximately equal to that of the standard Lanczos method given in Theorem~\ref{the:lanczos_convergence_stieltjes}.

The overall computational cost and storage requirements of the multi-shift CG method are determined by the number $p$ of poles of the rational approximation. These of course depend on the overall accuracy to which one wants to approximate $f(A)\vb$. The overall error of the multi-shift CG approximation $\vfMS_m$, 
$
\|f(A)\vb-\vfMS_m\|
$,
depends both on the accuracy of the rational approximation and the accuracy to which the shifted linear systems are solved. When aiming for an overall accuracy of $\varepsilon$, a (straightforward) approach is to construct a rational function $r$ such that 
$
|f(z)-r(z)| \leq \varepsilon/(2 \|\vb\|)  
$
for $z\in\spec(A)$, the spectral interval of $A$, 
and then solve the shifted linear systems accurately enough for fulfilling 
\begin{equation}\label{eq:accuracy_shifted_systems}
\|r(A)\vb-\vf_m^{MS}\| \leq \frac{\varepsilon}{2}
\end{equation}
such that overall
\begin{eqnarray*}
\|f(A)\vb - \vfMS_m\| &\leq& \|f(A)\vb-r(A)\vb\| + \|r(A)\vb-\vfMS_m\| \\
&\leq& \max\limits_{z \in \spec(A)} |f(z) - r(z)|\, \|\vb\| + \frac{\varepsilon}{2} \leq \varepsilon.
\end{eqnarray*}
\begin{remark}
Typically, systems associated with poles of large magnitude converge significantly faster than those corresponding to poles close to zero. Therefore one can employ a strategy for ``removing'' already converged systems from the iteration in order to not perform superfluous computations. Strategies for doing this without violating the condition~\eqref{eq:accuracy_shifted_systems} are discussed in~\cite[Section 5.3]{VanDenEshofFrommerLippertSchillingVanDerVorst2002} in the context of approximating the action of the matrix sign function.
\end{remark}

\subsection{Restarted Lanczos}\label{subsec:restarted_lanczos}

The idea of the restarted Lanczos method is to first compute an approximation \eqref{eq:lanczos_approximation} obtained from $m < m_{max}$ iterations of the standard Lanczos process, which we now denote by $\vf_m^{(1)}$, where the superscript is used to distinguish quantities belonging to different \emph{restart cycles}. The second cycle of the method then consists of an additive update
$
\vf_m^{(2)} = \vf_m^{(1)} + \ve_m^{(1)},
$ 
where $\ve_m^{(1)}$ is an approximation of the error $f(A)\vb - \vf_m^{(1)}$ obtained by $m$ new Lanczos iterations. Repeatedly applying this approach yields a sequence of approximations
\begin{equation*}
\vf_m^{(k)} = \vf_m^{(k-1)} + \ve_m^{(k-1)}, \quad k=2,3,\ldots \,
\end{equation*}
for $f(A)\vb$. In order to use the Lanczos method for approximating the error as $f(A)\vb - \vf_m^{(1)}=:e^{(1)}_m(A)\vv^{(1)}$ with a new function $e_m^{(1)}(z)$ and  a new vector $\vv^{(1)}$. First results in this direction were given in~\cite{EiermannErnst2006,IlicTurnerSimpson2010}, characterizing the restart function $e_m^{(1)}(z)$ as the $m$th order divided difference~\cite{DeBoor2005} of $f(z)$ with respect to the Ritz values, i.e., the eigenvalues of $T_m$. However, this error function representation turned out to be numerically unstable. We therefore cite a result from~\cite{FrommerGuettelSchweitzer2014a,FrommerGuettelSchweitzer2014b} which gives an integral representation for the error which is numerically stable and in addition useful for deriving theoretical results on the convergence of the restarted Lanczos method.

\begin{theorem}[Theorem~2.1 in~\cite{FrommerGuettelSchweitzer2014b}]\label{the:error_function}
Let $f$ be a Stieltjes function as in~\eqref{eq:stieltjes_function}. Assume $\spec(A) \cap (-\infty,0] = \emptyset$ and denote by $\vf_m$ the approximation \eqref{eq:lanczos_approximation} to $f(A)\vb$. Assume that $\spec(T_m) = \{\theta_1,\ldots,\theta_m\}$
satisfies $\spec(T_m) \cap (-\infty,0] = \emptyset$ and define
$$
e_m(z) :=  (-1)^{m+1} \|\vb\| \gamma_m \int_0^\infty \frac{1}{w_m(t)}\cdot \frac{1}{z+t} \dmu, \enspace z \not \in (-\infty,0],
$$
where $w_m(t) = (t+\theta_1)\cdots(t+\theta_m)$  and $\gamma_m = \prod_{i=1}^m \beta_{i+1}$.
Then
$$
f(A)\vb - \vf_m = e_m(A)\vv_{m+1},
$$
where $\vv_{m+1}$ is the $(m+1)$st Lanczos vector.
\end{theorem}

Theorem~\ref{the:error_function} recursively also holds for the Lanczos approximation resulting after a restart cycle because $e_m(z)$ is itself (a scalar multiple of) a Stieltjes function; see~\cite[Proposition~2.2]{FrommerGuettelSchweitzer2014b}. As a consequence, the error representation can be used to obtain a restarted Lanczos method with an arbitrary number of restart cycles. 
In~\cite{FrommerGuettelSchweitzer2014a,Schweitzer2016thesis} a version of this method was introduced which evaluates the error function $e_m(z)$ using adaptive numerical quadrature.
The following theorem (a more general version of Theorem~\ref{the:lanczos_convergence_stieltjes}) gives an upper bound on the error of the restarted Lanczos method.

\begin{theorem}[Theorem 4.3 in \cite{FrommerGuettelSchweitzer2014b}]\label{the:restarted_lanczos_convergence_stieltjes}
Let $A \in \CNN$ be Hermitian positive definite, $\vb \in \CN$, $f$ a Stieltjes function~\eqref{eq:stieltjes_function}, and $\vf_m^{(k)}$ the approximation from $k$ cycles of the restarted Lanczos method with restart length $m$. Further, let $\alpha_m(t)$ be defined as in \eqref{eq:c_alpha_cg} and let $t_0 \geq 0$ be the left endpoint of the support of $\mu$. Then
$$
\|f(A)\vb - \vf_m^{(k)}\|_A \leq C \alpha_m(t_0)^{k},
$$
where C is as in~\eqref{eq:constant_c} and $0 \leq \alpha_m(t_0) < 1$. In particular, the restarted Arnoldi method converges for all restart lengths $m \geq 1$.
\end{theorem}


\section{Rational Krylov methods for Stieltjes function}\label{sec:rational_methods}
In this section, we discuss two rational Krylov methods, namely the shift-and-invert Lanczos method and the extended Krylov method, for the approximation of Stietljes functions. The convergence of the shift-and-invert method is analysed in detail.  We also briefly comment on iteratively solving the linear systems arising in each iteration.

\subsection{Shift-and-invert Lanczos}\label{subsec:si_lanczos}

The shift-and-invert Lanczos method was introduced in~\cite{vdEH06} for ``preconditioning'' Lanczos iterations for the matrix exponential times a vector; see also~\cite{Moret2009} for related work. The main idea  is to replace $A$ by a matrix $B$ with  more favourable spectral properties, leading to faster convergence to $f(A)\vb$. In the following we give a short description of this method.


First, define $B = (A - \xi I)^{-1}$, where $\xi \in \R^-$ is an arbitrary, negative shift. How to best choose this parameter is discussed later in this subsection.
Applying $m$ steps of the Lanczos method to $B$ with starting vector $\vb$, we obtain the relation
\begin{equation}\label{eq:shift_and_invert_lanczos_relation}
B \VSI_m = \VSI_m\TSI_m + \beta^{SI}_m\vvSI_{m+1}\ve_m^T,
\end{equation}
where $\VSI_m$ now contains an orthonormal basis of $\spK_m(B,\vb).$ To extract an approximation for  $f(A)\vb$ from $\spK_m(B,\vb)$, we consider the  transformed function
\begin{equation}\label{eq:g}
g(y) := f(y^{-1} + \xi) \text{ where } y = (z - \xi)^{-1}.
\end{equation}
We have $f(A)\vb = g(B)\vb$ and define the \emph{standard} shift-and-invert approximation as 
$$
\vgSI_m := \|\vb\| \VSI_m g(\TSI_m)\ve_1  \approx f(A)\vb.
$$
For reasons that will become apparent when deriving the convergence bounds later in this section, we propose to instead use the so-called \emph{``corrected''} Lanczos approximation (first introduced in~\cite{Saad1992} for the approximation of $\varphi$-functions) 
\begin{equation}\label{eq:corrected_shift_and_invert_approximation}
 \vghatSI_m := \|\vb\| \VSI_m g(\TSI_m)\vehat_1 + \beta^{SI}_{m+1} (\ve_m^T g(\TSI_m)\ve_1) \vvSI_{m+1} \approx f(A)\vb.
\end{equation}
When $f$ is a Stieltjes function~\eqref{eq:stieltjes_function}, then the function $g$ defined in~\eqref{eq:g} clearly admits an integral representation
\begin{equation}\label{eq:g_integral_representation}
g(y) = y \int_0^\infty \frac{1}{1 + (\xi +t ) y}\dmu.
\end{equation}
Using~\eqref{eq:g_integral_representation}, we find the representation 
\begin{equation}\label{eq:error_function_si}
f(A)\vb - \vghatSI_m = e_m(B)\vvSI_{m+1},
\end{equation}
\vspace*{-5mm}
\begin{equation}\label{eq:error_function_si_integral}
e_m(y) :=   (-1)^{m+1} \|\vb\| \gamma_m y \int_0^\infty \frac{1}{w_m(t)}\cdot \frac{1}{1 + (\xi+t) y} \dmu, \quad y \in (0,-1/\xi),
\end{equation}
for the error of the corrected shift-and-invert approximation~\eqref{eq:corrected_shift_and_invert_approximation}, where $\gamma_m$ and $w_m$ are as in Theorem~\ref{the:error_function}. In other words,
\begin{eqnarray}
 e_m(B)\vvSI_{m+1} &=& (-1)^{m+1} \|\vb\|  \gamma_m B \int_0^\infty \frac{1}{w_m(t)}\cdot (I + (\xi+t) B)^{-1} \vvSI_{m+1} \dmu \nonumber\\
		 &=& B  \int_0^\infty \ve_m(t)  \dmu,\label{eq:shifted_systems_si}
\end{eqnarray}
where $\ve_m(t)$ is the error of the $m$th CG approximation to the linear system 
$
 (I + (\xi+t) B) \vx(t) = \vb.
$ 
This follows from the fact that 
\[ 
     \frac{(-1)^{m+1}  \gamma_m \| \vb \|}{w_m(t)} \vvSI_{m+1} = \vr_m(t)
\]
 is the residual of the $m$th CG approximation $\vx_m(t)$ for that linear system.

In the following, we use the error function representation~\eqref{eq:error_function_si_integral} to derive results on the speed of convergence of the shift-and-invert method and on how to choose the shift $\xi$. Similar results have  been obtained in~\cite{Moret2009}, but we provide a different proof here which yields  explicit constants in the bounds that were previously unavailable. 

\begin{lemma}\label{lem:error_estimate_general2}
Let $A \in \CNN$ be Hermitian positive definite, $B = (A - \xi I)^{-1}$, $\vb \in \CN$, $g$ a function of the form \eqref{eq:stieltjes_function},  and $\vghatSI_m$ as defined in \eqref{eq:corrected_shift_and_invert_approximation}. Let $\lmin$ and $\lmax$ denote the smallest and largest eigenvalue of $A$, respectively, and define  
\begin{equation}\label{eq:aux_functions2}
\kappa_\xi(t) = \begin{cases}
\frac{\lmax+t}{\lmin+t}\cdot \frac{\lmin-\xi}{\lmax-\xi} & \!\!\!\!\!\text{if } t \leq -\xi \\[2mm]
\frac{\lmin+t}{\lmax+t}\cdot \frac{\lmax-\xi}{\lmin-\xi} & \!\!\!\!\!\text{if } t > -\xi \\
\end{cases},\ \ c_\xi(t) = \frac{\sqrt{\kappa_\xi(t)}-1}{\sqrt{\kappa_\xi(t)}+1}, \ \ \alpha_m^\xi(t) = \frac{1}{\cosh(m \ln c_\xi(t))}.
\end{equation}
Then the norm of the error of $\vghatSI_m$ is bounded by
\begin{eqnarray}\label{eq:error_estimate_general_si}
 \|f(A)\vb - \vghatSI_m \| &\leq& \|\vb\| \sqrt{\frac{\lmax-\xi}{\lmin - \xi}} \int_0^{-\xi} \frac{\alpha_m^\xi(t)}{\lmin + t}\dmu \nonumber\\
 & &+ \frac{\lmax - \xi}{\lmin - \xi} \int_{-\xi}^\infty \frac{\alpha_m^\xi(t)}{\sqrt{\lmin + t}\sqrt{\lmax+t}}\dmu.
\end{eqnarray}
\end{lemma}

\begin{proof}
By using~\eqref{eq:error_function_si} and~\eqref{eq:shifted_systems_si}, we can write
$
f(A)\vb - \vghatSI_m = B \int^\infty_0 \ve_{m}(t) \dmu,
$ 
where $\ve_{m}(t)$ denotes the error of the approximation $\vx_{m}(t)$ from $m$ steps of CG for the shifted linear system $(I + (\xi + t) B)\vx = \vb$. This yields
\begin{eqnarray}
\|f(A)\vb - \vghatSI_m\|_B &\leq& \| B\|_B \int_0^\infty \|\ve_m(t)\|_B \dmu \nonumber\\
                            &\leq& \frac{1}{\lmin - \xi} \int_0^\infty \frac{1}{\sqrt{\lmin + t}}\|\ve_m(t)\|_{I + (\xi+t)B} \dmu, \nonumber 
\end{eqnarray}
where we used that $\|\vv\|_B \leq 1/\sqrt{\lmin + t} \cdot \|\vv\|_{I + (\xi+t)B}$ holds for all $t \in [0,\infty)$ and that $\|B\|_B = 1/(\lmin - \xi)$. We now apply Theorem~\ref{the:cg_convergence} for the shifted matrices $I+(\xi+t)B$, which are positive definite for $t \in [0,\infty)$. Note that $\kappa_\xi(t)$ is exactly the condition number of the shifted matrix $I + (\xi+t)B$. Applying the CG estimate for all $t$ and using the fact that the initial guess  is $\vx_0(t) = \vnull$ for all $t$, we conclude that
\begin{equation}\label{eq:proof_lemma42}
\|f(A)\vb - \vghatSI_m\|_B \leq \frac{1}{\lmin - \xi} \int_0^\infty \frac{\alpha_m^\xi(t)}{\sqrt{\lmin + t}}  \|\vx^\ast(t)\|_{I + (\xi+t)B} \dmu.
\end{equation}
As $\vx^\ast(t) = (I + (\xi+t)B)^{-1}\vb$, a straightforward calculation shows that 
\begin{equation}\label{eq:proof_lemma52}
\|\vx^\ast(t)\|_{I + (\xi+t)B} \leq \|\vb\| \begin{cases}
                                               \frac{\sqrt{\lmin - \xi}}{\sqrt{\lmin + t}}  &  \text{if } t \leq -\xi, \\[2mm]
					       \frac{\sqrt{\lmax - \xi}}{\sqrt{\lmax + t}}  &  \text{if } t > -\xi. \\
                                              \end{cases}
\end{equation}
Inserting \eqref{eq:proof_lemma52} into \eqref{eq:proof_lemma42}, using the fact that $\|\vv\| \leq \sqrt{\lmax-\xi}\,\|\vv\|_B$ for all $\vv\in\CN$ and splitting the integral at $-\xi$ completes the proof.\hfill
\end{proof}

The error estimate~\eqref{eq:error_estimate_general_si} shows that the asymptotic convergence factor of the corrected Lanczos approximation for $g(B)\vb$ will be determined by the largest asymptotic CG convergence factor $\alpha_m^\xi(t)$ across all shifts $t \in [0,\infty)$. According to~\eqref{eq:aux_functions2}, the values $\alpha_m^\xi(t)$ also depend on the shift $\xi$ and we will therefore now determine the value of $\xi$ for which the maximum of $\alpha_m^\xi(t)$ becomes smallest possible. 

For this, first note that $c_\xi$ increases monotonically as a function of $\kappa_\xi(t)$ and $\alpha_m^\xi$ increases monotonically as a function of $c$. Therefore, $\alpha_m^\xi(t)$ attains its largest value where $\kappa_\xi(t)$ attains its largest value. The function $\kappa_\xi(t)$ is monotonically decreasing on $[0,-\xi]$ and monotonically increasing on $[-\xi,\infty)$. Therefore,
\begin{equation}\label{eq:alpha_bound}
\alpha_m^\xi(t) \leq \max\{\kappa_\xi(0), \kappa_{\xi,\infty}\}, \text{ where } \kappa_{\xi,\infty} = \lim_{t\rightarrow \infty} \kappa_\xi(t).
\end{equation}
It depends on the choice of the shift $\xi$ which of the two values $\kappa_\xi(0)$ and $\kappa_{\xi,\infty}$ is larger. We have
$$
\kappa_\xi(0) = \frac{\lmax}{\lmin}\cdot \frac{\lmin-\xi}{\lmax-\xi} \text{ and } \kappa_{\xi,\infty} = \frac{\lmax-\xi}{\lmin-\xi},
$$
i.e., $\kappa_{\xi,\infty}$ is monotonically increasing in $\xi$ and $\kappa_\xi(0)$ is monotonically decreasing in~$\xi$. Hence the bound for $\alpha_m^\xi(t)$ in~\eqref{eq:alpha_bound} is minimal if $\xi$ is chosen so that $\kappa_\xi(0) = \kappa_{\xi,\infty}$, i.e.,
\begin{equation}\label{eq:kappa2}
\frac{\lmax}{\lmin}\cdot \frac{\lmin-\xi}{\lmax-\xi}=\frac{\lmax-\xi}{\lmin-\xi}.
\end{equation}
Equation~\eqref{eq:kappa2} is solved by the shift
\begin{equation}\label{eq:xi_opt}
\xi = -\sqrt{\lmin\cdot\lmax}.
\end{equation} 
Using the shift~\eqref{eq:xi_opt}, we find the following error bound.
\begin{theorem}\label{the:shift_invert}
Let the assumptions of Lemma~\ref{lem:error_estimate_general2} hold and let $\xi = -\sqrt{\lmin\lmax}$. Define the functions
\begin{equation}\label{eq:f1f2}
f_1(z) = \int_0^{-\xi} \frac{1}{z+t}\dmu \text{ and } f_2(z) = \int_{-\xi}^0 \frac{1}{z+t} \dmu.
\end{equation}
Then
$$
 \|g(B)\vb - \vghatSI_m \| \leq \|\vb\| \sqrt{\frac{\lmax}{\lmin}} \left( \sqrt{\frac{\lmax}{\lmin}} f_1(\lmin)+ f_2(\sqrt{\lmin\lmax})\right) \alpha_m(0).
$$
\end{theorem}
\begin{proof}
By the considerations above, we have that for $\xi = - \sqrt{\lmin\lmax}$ the CG convergence factors fulfill $\alpha_m^\xi(0) \geq \alpha_m^\xi(t)$ for all $t \in [0,\infty)$. In addition, we have
\begin{equation}\label{eq:condition_shift}
\frac{\lmax-\xi}{\lmin-\xi} = \frac{\sqrt{\lmax}\left(\sqrt{\lmax}+\sqrt{\lmin}\right)}{\sqrt{\lmin}\left(\sqrt{\lmax}+\sqrt{\lmin}\right)} = \sqrt{\frac{\lmax}{\lmin}}.
\end{equation}
Using the fact that $\sqrt{\lmin + t}\sqrt{\lmax + t} \leq \sqrt{\lmin\lmax} + t$ after inserting~\eqref{eq:condition_shift} into~\eqref{eq:error_estimate_general_si} concludes the proof.\hfill
\end{proof}

\subsection{Extended Krylov}\label{subsec:extended_krylov}
The extended Krylov subspace method~\cite{DruskinKnizhnerman1998,JagelsReichel2011,KnizhnermanSimoncini2010}, like the shift-and-invert method, is a special case of a rational Krylov  method. Here the approximations to $f(A)\vb$ are extracted from an extended Krylov space
$$
\spEK_m(A,\vb) = \Span\{\vb,A\vb,\dots,A^{m-1}\vb, A^{-1}\vb,\dots,A^{-m}\vb\}.
$$
One iteration of the method involves adding one basis vector from the ``positive'' and one vector from the ``negative'' sequence. A five-term recursion for the basis vectors was first derived in~\cite{Simoncini2007}. Recursion relations for more general extended Krylov sequences are treated in~\cite{JagelsReichel2009,JagelsReichel2011}. The method yields an \emph{extended Lanczos decomposition}
$$
AV_m^{EK} = V_m^{EK}T_m^{EK} + [\vv_{2m+1},\vv_{2m+2}]\tau_{m}[\ve_{2m-1},\ve_{2m}]^T,
$$
where $V_m^{EK} = [\vv_1,\dots,\vv_{2m}]\in\C^{N \times 2m}$ contains an orthonormal basis of $\spEK_m(A,\vb)$, $T_m^{EK} = (V_m^{EK})^HAV_m^{EK} \in \C^{2m \times 2m}$ and $\tau_m = [\vv_{2m+1},\vv_{2m+2}]^HA[\vv_{2m+1},\vv_{2m+2}] \in \C^{2\times 2}$. An approximation for $f(A)\vb$ can then be obtained by projection onto the extended Krylov space in the usual way, i.e.,
\begin{equation}\label{eq:extended_lanczos_approximation}
f(A)\vb \approx \vf_m^{EK} = \|\vb\|V_m^{EK}f(T_m^{EK})\ve_1.
\end{equation}
Using the algorithmic approach from~\cite{Simoncini2007}, one iteration of the extended Krylov subspace method requires one matrix-vector product with $A$, the solution of one linear system with $A$, and the computation of six inner products/vector norms in the orthonormalization process. The matrix $T_m^{EK}$ can be cheaply computed from the orthonormalization coefficients without any further inner products~\cite{JagelsReichel2009}.

The convergence of the extended Krylov methods for the approximation of Stieltjes matrix functions via~\eqref{eq:extended_lanczos_approximation} has been analyzed in~\cite{DruskinKnizhnerman1998,KnizhnermanSimoncini2010,beckermann2009error}. Here we will use the following result.

\begin{theorem}[Section~6.1 in~\cite{beckermann2009error}]\label{the:extended_krylov_convergence}
Let $A \in \CNN$ be Hermitian positive definite with $\spec(A) = [\lmin,\lmax] \subset \R^+$, let $f$ be a Stieltjes function~\eqref{eq:stieltjes_function} and  $\vf_m^{EK}$ be the $m$th extended Krylov approximation~\eqref{eq:extended_lanczos_approximation} for $f(A)\vb$. Then 
$$
\|f(A)\vb - \vf_m^{EK}\| \leq C \left(\frac{\sqrt[4]{\kappa}-1}{\sqrt[4]{\kappa}+1}\right)^{2m} 
\simeq \mathcal{O}\left(\exp\left(-4m\sqrt[4]{\frac{\lmin}{\lmax}}\right)\right),
$$
for a constant $C>0$ and with the term on the right asymptotically sharp for large $\kappa = {\lmax}/{\lmin}$.
\end{theorem}

\begin{remark}
Theorem~\ref{the:extended_krylov_convergence}  states that the convergence factor of the extended Krylov method when expressed in terms of the subspace dimension $\dim\spEK_m = 2m$ (instead of the order $m$), is the same as 
$$\alpha(0) = \frac{\sqrt[4]{\kappa}-1}{\sqrt[4]{\kappa}+1},$$
just like the shift-and-invert method discussed in Section~\ref{subsec:si_lanczos}. 
\end{remark}

\subsection{Polynomial solves in the inner iteration}\label{subsec:inner_outer_methods}
Both the shift-and-invert method and the extended Krylov subspace method require the solution of a linear system in each iteration. They are therefore particularly attractive for matrices for which direct solution methods can be efficiently applied (e.g., not too large matrices with rather small bandwidth, for which it is feasible to compute a Cholesky decomposition). In particular, as the poles of the rational Krylov subspace stay the same across all iterations, it suffices to compute a Cholesky decomposition (of $A$ or $A-\xi I$, depending on the method) once and then use it in all subsequent iterations.

We are interested in the very large scale case, in which only a few vectors of length $N$ can be stored at the same time, meaning particularly that computing a Cholesky decomposition of $A$ is not an option. Therefore, the linear systems occurring in the shift-and-invert or extended Krylov method can only be solved approximately by an iterative method (the \emph{inner iteration}). We deal here with the case in which the inner iteration is again a Krylov subspace method. As $A$ is Hermitian positive definite, a natural choice is the conjugate gradient method.

In cases where a direct solver can be used, one iteration of the shift-and-invert Krylov method and one iteration of the extended Krylov method have approximately the same cost (as there is no difference in the cost of computing a Cholesky factorization of $A$ or $A-\xi I$). When using an iterative method, this dramatically changes, however. In the shift-and-invert method, with the optimal shift $\xi = -\sqrt{\lmin \lmax}$ it follows from~\eqref{eq:condition_shift} that
$
 \kappa(A-\xi I) = \sqrt{\kappa(A)}.
$
Thus, without preconditioning of the inner iteration, we can expect the linear systems within the extended Krylov  method to be much more difficult to solve and require a significantly higher number of iterations.

An important question that arises in the context of using an inner iteration for the solution of the linear systems in a rational Krylov method is to which accuracy these systems need to be solved in order to not negatively influence the convergence of the outer iteration. We discuss this topic in detail in the next section.

\section{Relaxing the tolerance in outer-inner rational Krylov methods}\label{sec:inexact_iterations}

In this section, we only discuss the shift-and-invert method. For the extended Krylov method, very similar results can be formulated in a straight-forward manner. We have so far only considered the error of the Lanczos approximation for $g(B)\vb$.  If we want to use an inexact version of $B = (A - \xi I)^{-1}$ by solving the involved linear systems iteratively, we need to modify \eqref{eq:shift_and_invert_lanczos_relation} to
\begin{equation}\label{eq:arnoldi_decomposition_inexact}
B (V_m - R_m) = V_{m} H_m + \beta_m \vv_{m+1} \ve_m^T,
\end{equation}
where each column $\vr_j$ of $R_m$ is the residual incurred when solving for $\vw = (A - \xi I)^{-1} \vv_{j}$. Note that we have replaced the tridiagonal matrix $T_m^{SI}$ by a generally dense upper-Hessenberg matric $H_m$. Defining $E_m := -B R_m V_m^*$, a matrix of rank at most~$m$, we can further rewrite \eqref{eq:arnoldi_decomposition_inexact} as
$$
(B + E_m) V_m = V_{m} H_m + \beta_m \vv_{m+1} \ve_m^T.
$$
Therefore, when inexact inner iterations are used, we are effectively computing an Arnoldi approximation $\vg_m := \|\vb\| V_m g(H_m)\ve_1$ to  $g(B+E_m)\vb$, not $g(B)\vb = f(A)\vb$ as intended. It is clear that  $\| f(A)\vb - \vg_m \| = \| g(B) \vb - \vg_m \|$ and 
\begin{equation}
 \| g(B) \vb - \vg_m \| \leq   \| g(B+E_m)\vb - \vg_m\| + \| g(B)\vb - g(B+E_m)\vb \|,\label{eq:twoterms}
\end{equation}
and so we need to control both terms on the right-hand side. 

\subsection{The first term $\| g(B+E_m)\vb - \vg_m\|$} The first term corresponds to the error of the (exact) Arnoldi approximation for $g(B+E_m)\vb$ and it is analysed in Appendix~\ref{sec:inexact_convergence}. We show that for $E_m$ small enough, this Arnoldi approximation still converges at a rate very close to that of the unperturbed Krylov approximation. 

By Theorem~\ref{the:shift_invert} we know that the exact shift-and-invert method with optimal shift $\xi=-\sqrt{\lmin \lmax}$, as well as the standard Lanczos and extended Krylov method, converge geometrically as
\begin{equation}
	\| f(A)\vb - \vf_m \| \leq C \alpha^m,
	\label{eq:assum}
\end{equation}
where $C$ is some constant independent of $m$, and 
\begin{equation}
\alpha = \begin{cases} 
\frac{\sqrt{\kappa(A)} - 1}{\sqrt{\kappa(A)}+1} &\mbox{for standard Lanczos, } \\
\frac{\sqrt[4]{\kappa(A)} - 1}{\sqrt[4]{\kappa(A)}+1} &\mbox{for shift-and-invert Lanczos and extended Krylov. } \\
\end{cases} 
\label{eq:alpha}
\end{equation}
While  Theorem~\ref{the:shift_invert} gives an expression for the constant $C$ needed to strictly satisfy the bound \eqref{eq:assum}, a potentially sharper estimate can be obtained by expanding  $f(A)\vb = \sum_{j=1}^N \gamma_j^{(N)}\vv_j$ into a complete orthonormal Krylov basis $\{\vv_j\}$ of $\mathbb{C}^N$ (assuming that the invariance index of the Krylov space is $N$; if this is not the case, a reduced expansion of $f(A)\vb$ can be used for the same argument). Likewise, let us write $\vf_m = \sum_{j=1}^m \gamma_j^{(m)}\vv_j$. It can be shown that each of the coefficients $\gamma_j^{(m)}$ approaches $\gamma_j^{(N)}$ monotonically as $m$ increases, in the sense that 
$
	|\gamma_j^{(1)}| \leq |\gamma_j^{(2)}| \leq \cdots \leq |\gamma_j^{(N)}|.
$ 
This is true for the standard Lanczos method, which has been known since~\cite{Druskin2008, Frommer2009}, but also for the extended Krylov~\cite{Schweitzer2016} and even the shift-and-invert method. It is now easy to derive an upper bound on $|\gamma_j^{(N)}|$, and thereby all $|\gamma_j^{(m)}|$. Using \eqref{eq:assum} we have
\[
	|\gamma_m^{(N)}|^2 \leq 
	 \sum_{j=1}^{m-1} |\gamma_j^{(N)} - \gamma_{m-1}^{(N)}|^2  
	+ \sum_{j=m}^{N} |\gamma_j^{(N)}|^2 
	= 	
	\| f(A)\vb - \vf_{m-1} \|^2 
	\leq \left(\frac{C}{\alpha}\right)^2 \alpha^{2m},
\]
i.e., the Krylov coefficients $\gamma_j^{(m)}$ used to form $\vf_m$ also decay geometrically at a rate~$\alpha$; see also~\cite{PozzaSimoncini2019} for related results. Let us use the model that the unperturbed approximation $\vf_m$ has coefficients that satisfy a geometric series  and therefore  
\[
	\| f(A)\vb \|^2 = \| \vf_N \|^2  \approx \left(\frac{C}{\alpha}\right)^2 \sum_{j=1}^N  \alpha^{2j}  
	\approx 
	\left(\frac{C}{\alpha}\right)^2 \frac{\alpha^2}{1-\alpha^2},
\]	
suggesting the estimate  
\begin{equation}\label{eq:estimate_C}\
C \approx \sqrt{1-\alpha^2}\, \|f(A)\vb\|.
\end{equation}

\subsection{The second term $\| g(B)\vb - g(B+E_m)\vb \|$}
In order to analyse the second term on the right-hand side of \eqref{eq:twoterms}, we first consider the simplified  function $g_t(y) = (y^{-1} + \xi + t)^{-1}$. 
Thanks to the Sherman--Morrison--Woodbury formula we have 
\begin{align*}
(B + E_m)^{-1} 
& = B^{-1} + B^{-1} B R_m (I - V_m^H B^{-1} B R_m)^{-1} V_m^H B^{-1}  \\
& = B^{-1} + R_m (I - V_m^H R_m)^{-1} V_m^H B^{-1},
\end{align*}
and therefore
\begin{align*}
g_t(B + E_m)   
& = [(B+E_m)^{-1} + (\xi + t)I]^{-1}  \\
& = [ B^{-1} + R_m (I - V_m^H R_m)^{-1} V_m^H B^{-1} + (\xi + t)I]^{-1} \\
& = [ B^{-1} + \underbrace{R_m (I - V_m^H R_m)^{-1}}_{=:U} \underbrace{V_m^H B^{-1}}_{=:V} + (\xi + t)I]^{-1} \\
& = g_t(B) - g_t(B) U(I+V g_t(B) U)^{-1} V g_t(B).
\end{align*}
As a consequence,  
\[
\|g_t(B + E_m)  - g_t(B)\| \leq \| g_t(B) \| \| U(I+V g_t(B) U)^{-1} \| \| V g_t(B) \|.
\]
Since $B=(A-\xi I)^{-1}$ is Hermitian, it is easy to see that $\|g_t(B)\|=1/(\lmin + t)$ and 
\[
\| V g_t(B) \| = \|V_m^H (I + (\xi+t)B)^{-1} \| \leq \|(I + (\xi+t)B)^{-1} \| = C_t,
\]
where
\[
C_t := 
\begin{cases}
\frac{\lmin-\xi}{\lmin + t}, \quad t < -\xi, \\[2mm]
\frac{\lmax-\xi}{\lmax + t}, \quad t \geq -\xi.
\end{cases}
\]
To estimate $\| U(I+V g_t(B) U)^{-1} \|$, we assume that $\| R_m\|\ll 1$ so that upon using a truncated Neumann series  $U = R_m (I- V_m^H R_m)^{-1} \approx R_m(I + V_m R_m) \approx R_m$ and $\| U(I+V g_t(B) U)^{-1} \|\approx R_m$. This gives the approximate error bound
\begin{equation}\label{eq:estimate_gt}
\|g_t(B + E_m)  - g_t(B)\| \lesssim \| R_m \| \frac{C_t}{\lmin + t}.
\end{equation}
As we have
$g(z) = \int_0^\infty g_t(z) \dmu,$
we  obtain an approximate error bound for $g(B+E_m)$ by integrating~\eqref{eq:estimate_gt}:
\begin{eqnarray}
\|g(B+E_m)-g(B)\| &=& \left\|\int_0^\infty g_t(B+E_m) - g_t(B) \dmu \right\| \nonumber\\
									&\leq& \int_0^\infty \|g_t(B+E_m) - g_t(B) \| \dmu \nonumber\\
									&\lesssim& \| R_m \| \int_0^\infty\frac{C_t}{\lmin + t}\dmu .\label{eq:estimate_integral}
\end{eqnarray}
To rewrite the right-hand side of~\eqref{eq:estimate_integral}, we use the same techniques as in the proof of Lemma~\ref{lem:error_estimate_general2} and Theorem~\ref{the:shift_invert}. Let the functions $f_1,f_2$ be defined as in~\eqref{eq:f1f2}, then
\begin{eqnarray*}
\int_0^\infty \frac{C_t}{\lmin+t} \dmu &=& (\lmin-\xi)\int_0^{-\xi} \frac{1}{(\lmin+t)^2}\dmu \\
& &+ (\lmax-\xi)\int_0^{-\xi} \frac{1}{(\lmin+t)(\lmax+t)}\\
&\leq& (\lmin-\xi)\int_0^{-\xi} \frac{1}{(\lmin+t)^2}\dmu \\
& & + (\lmax-\xi)\int_0^{-\xi} \frac{1}{(\sqrt{\lmin\lmax}+t)^2}\\
&=& (\lmin-\xi)\left|f_1^\prime\left(\lmin\right)\right| + (\lmax-\xi)\left|f_2^\prime\left(\sqrt{\lmin\lmax}\right)\right|. 
\end{eqnarray*}
Unfortunately, no closed form of the functions $f_1,f_2$ or their derivatives is available. One can thus either evaluate the integrals  numerically or use the trivial upper bound $|f_1^\prime(z)|, |f_2^\prime(z)| \leq f^\prime(z)$. Using the latter approach, we obtain the final estimate
\begin{equation}\label{eq:estimate_second_term}
\|g(B+E_m)-g(B)\| \lesssim \|R_m\| \left((\lmin-\xi)|f^\prime(\lmin)| + (\lmax-\xi)|f^\prime(\sqrt{\lmin\lmax})|\right).
\end{equation}

\section{Theoretical and practical comparison of the different methods}\label{sec:comparison}
We now devise recommendations which of the methods discussed in Section~\ref{sec:polynomial_methods} and~\ref{sec:rational_methods} are best suited for approximating $f(A)\vb$ in a given situation (of available memory, conditioning of $A$, size of $A$, availability of spectral information) based purely on the theoretical results available for these methods. Additionally, we summarize several other features and (dis)advantages of the different methods in a concise manner and perform numerical experiments in order to gauge whether the predictions obtained from the theoretical results are trustworthy.

\subsection{Advantages, disadvantages, and prerequisites}\label{subsec:general_comparison}

\begin{table}
\renewcommand{\arraystretch}{1.3}
\centering\begin{tabular}{l|c|c|c|c|c}
 & 2PL & MSCG & R.~Lan. & EKSM & SI \\
\hline
accuracy limited by memory & \xmark & \cmark & \xmark & \cmark & \cmark  \\
\hline
requires spectral information & \xmark & \cmark &  \xmark &  \xmark & \cmark \\
\hline
preconditioning possible &  \xmark &  \xmark &  \xmark & \cmark & \cmark \\
\hline
additional overhead & \cmark &\cmark &  \xmark & \xmark & \xmark \\
\hline
inner conv.\ factor det.\ by &  --- &  --- &  --- & $\sqrt{\kappa(A)}$ & $\sqrt[4]{\kappa(A)}$\\
\hline
conv.\ factor det.\ by & $\sqrt{\kappa(A)}$ & $\sqrt{\kappa(A)}$ & $\sqrt{\kappa(A)}$ & $\sqrt[4]{\kappa(A)}$ &  $\sqrt[4]{\kappa(A)}$\\
\end{tabular}

\vspace{.1cm}
\caption{Overview of the general advantages, disadvantages and prerequisites of two-pass Lanczos (2PL), multi-shift CG (MSCG), restarted Lanczos (R.~Lan.), extended Krylov (EKSM) and shift-and-invert Lanczos (SI), together with the quantities governing the (worst-case) speed of convergence of the method (and in case of the outer-inner methods also of the inner iteration).}\label{tab:comparison}
\end{table}

We briefly discuss general properties of the different methods, which go beyond the comparison of error bounds  in Section~\ref{subsec:prediction_matvec}. This comparison is compactly summarized in Table~\ref{tab:comparison}, together with the quantities determining the asymptotic speed of convergence.

\paragraph{Limited accuracy due to available memory}
Without further countermeasures, the accuracy of some of the presented methods is still limited by the available memory. For the multi-shift CG method, the available memory dictates the maximum number of poles which can be used for the rational approximation, as one or two additional vectors of length $N$ (depending on the specific implementation, cf.~Remark~\ref{rem:reduce_memory_mscg}) need to be stored. If the number of poles necessary for reaching the target accuracy exceeds the available memory, an alternative is to run the multi-shift CG method several times for subsets of the poles, which then increases the number of  matrix-vector and inner products.

For the extended Krylov and shift-and-invert Lanczos method, the attainable accuracy is limited by the available memory because the outer iteration still requires the storage of the full orthonormal basis in order to construct the final approximation to $f(A)\vb$. If this becomes a limiting factor, restarting  techniques could be employed for the outer iteration. However, the restarting of rational Krylov methods and the interaction between restarts and inexact inner solves are largely unexplored topics so far.

The two-pass Lanczos and restarted Lanczos method can reach any desired accuracy independent of the available memory (ignoring numerical effects like round-off error and assuming that the tridiagonal $m \times m$ matrix in the two-pass Lanczos method does not grow beyond memory).

\paragraph{Reliance on a-priori spectral information} In two of the discussed methods, spectral information on the matrix $A$ is required. In the multi-shift CG method, when constructing a suitable rational approximation $r\approx f$, one requires (bounds on) the largest and smallest eigenvalue of $\lmax$ and $\lmin$ of $A$. In the same way, computing the ``optimal'' shift $\xi = -\sqrt{\lmin\lmax}$ in the shift-and-invert method requires knowledge of these extremal eigenvalues. It is however possible to choose an arbitrary pole $\xi$ independent of spectral information of $A$. For certain functions such alternative strategies have been shown to be successful, see, e.g.~\cite{MoretNovati2019}.

In contrast, the two-pass Lanczos, restarted Lanczos and extended Krylov method do not require any a priori spectral information.

\paragraph{Possibility for preconditioning}
By applying a suitable preconditioner, the number of iterations necessary in the outer-inner methods, i.e., the extended Krylov and shift-and-invert Lanczos method, can potentially be greatly reduced. This can make these methods much more competitive than what one would expect from the inner convergence factors shown in Table~\ref{tab:comparison}. However, with preconditioning, the extracted approximations are no longer elements of a polynomial Krylov space $\spK_m(A,\vb)$, and hence a fair comparison is no longer possible. 

\paragraph{Additional overhead} 
Besides matrix-vector products, there are also other arithmetic operations that add to the computational  complexity of the considered methods. For the two-pass Lanczos method, $f(T_m)$ needs to be evaluated and, when $m$ gets large, this can be a challenge in its own right. While this is in principle also true for the extended Krylov and shift-and-invert method, the number of (outer) iterations in these methods will typically be significantly smaller. In the multi-shift CG method, additional vector operations for each pole beyond the first have to be performed. When the target accuracy is increased, this will also lead to an increase in the degree of the rational approximation and thus the number of poles. Depending on how the cost of a matrix-vector product compares to a vector operation, this additional work can become non-negligible; see also Experiment~\ref{experiment:work_units} below.

\subsection{Predicting the number of matrix-vector products}\label{subsec:prediction_matvec}

We now use the convergence results from above to estimate the number of matrix-vector products that are needed to achieve a certain relative error when approximating $f(A)\vb$, and thus obtain recommendations for which methods are most suitable under which circumstances. Let us stress that all the bounds presented so far are worst-case predictions that only take the extremal eigenvalues into account and therefore cannot, e.g., predict superlinear convergence effects due to spectral adaption \cite{BeckermannGuettel2012}. Therefore, our  predictions  cannot be expected to be accurate for all matrices with $\spec(A)\subseteq [\lmin,\lmax]$, but rather only for matrices whose spectra are close to a worst-case distribution. Put another way, our predictions can be expected to be good in regimes where $m/N$ is small so that superlinear convergence has not set in. As most of the discussed methods are influenced by this in a similar manner, we still hope that the recommendations derived from this worst-case analysis are also valid in other cases. This is indeed confirmed by the numerical experiments reported in Section~\ref{subsec:experiments}.

For obtaining the predictions in the non-restarted methods, we first estimate the number of (outer) iterations via the relation
\begin{equation}\label{eq:error_reduction}
\|\vf_m - f(A)\vb\| \leq C \alpha^m,
\end{equation}
where $\alpha$ is given by~\eqref{eq:alpha} and $C$ by~\eqref{eq:estimate_C}. Here $\vf_m$ can be the $m$th iterate of the two-pass Lanczos, shift-and-invert Lanczos, or extended Krylov method, or the $m$th conjugate gradient iterate for the system $(A - \zeta_1I)\vx_1 = \vb$, where $\zeta_1$ is the pole with smallest absolute value. From~\eqref{eq:error_reduction} we then obtain the prediction
$$
m^\ast = \left\lceil \log_{\alpha}\left(\frac{\varepsilon}{C}\right) \right\rceil.
$$
For the  two-pass Lanczos method, the estimated number of matrix-vector products is $2m^\ast$, while for the multi-shift CG method it is $m^\ast$ (the value of $m^\ast$ differs here, as the conditioning of the matrix $A - \zeta_i I$ is slightly better than that of $A$). For the outer-inner rational-polynomial methods, after computing the necessary number of outer iterations, we use the same approach for estimating the inner iterations. By summing over all inner iterations, we then obtain an estimate of the total number of matrix-vector products. 

For the restarted Lanczos method, after $k$ cycles with restart length $m_{re}$ we have
$$
\|\vf_{m_{re}}^{(k)} - f(A)\vb\| \leq C (\alpha_{m_{re}}(t_0))^k,
$$
with the convergence factor $\alpha_{m_{re}}(t_0)$ given in~\eqref{eq:aux_functions}. We obtain the estimate
\begin{equation}\label{eq:estimate_cycle_number}
k^\ast = \left\lceil \log_{\cosh(m_{re}\log(\alpha))}(\cosh(m^\ast\log_{\alpha})) \right\rceil,
\end{equation}
where $\alpha$ and $m^\ast$ correspond to the standard Lanczos method and where we have used the representation of the Lanczos and restarted Lanczos convergence factor in terms of the hyperbolic cosine. This is due to the fact that the estimate
$$\frac{1}{\cosh m_{re}\ln c} \approx 2 c^{m_{re}}$$
is rather rough for $m_{re}$ small. As the restart length is typically a small value, we thus obtain much better estimates using~\eqref{eq:estimate_cycle_number}. The final estimate for the number of matrix-vector products is then given by $m_{re}\cdot k^\ast$.

\subsection{Experimental confirmation of the predictions}\label{subsec:experiments}

We now perform   numerical experiments to illustrate how reliable the predictions from the previous section are in practice. It turns out that choosing the target residual norms as suggested by~\eqref{eq:estimate_second_term} in the inner iterations of the rational methods is much stricter than necessary to reach the desired accuracy. Experimentally, we found the following strategy to yield sufficient accuracy in our experiments: when the overall target accuracy is $\varepsilon$, we solve the first linear system to a residual norm below
$$\varepsilon_1 = \frac{\varepsilon}{2((\lmin-\xi)|f^\prime(\lmin)| + (\lmax-\xi)|f^\prime(\sqrt{\lmin\lmax})|)}$$
and let the residual norm grow geometrically, $\varepsilon_j = \frac{\varepsilon_1}{\alpha(0)^{j-1}}, j = 2, 3,\dots$. Apart from this deviation from our theoretical basis, all methods are executed as described before.

\begin{table}
\renewcommand{\arraystretch}{1.3}
\centering\begin{tabular}{l|c|c|c|c|c}
 & 2PL & MSCG & R.~Lan. & EKSM & SI \\
\hline
Predicted matrix-vector products & 564 & 215 & 501 & 1800 & 5729  \\
\hline
Required matrix-vector products & 552 & 240 &  480 &  1883 & 6903 \\
\end{tabular}

\vspace{.1cm}
\caption{Predicted number of matrix-vector products for two-pass Lanczos (2PL), multi-shift CG (MSCG), restarted Lanczos (R.~Lan.), extended Krylov (EKSM), and shift-and-invert Lanczos (SI), together with the number actually required. The matrix $A$ has Chebyshev eigenvalues in $[0.1, 200.1]$.}\label{tab:experiment1}
\end{table}

\begin{experiment}\label{experiment:chebyshev}
In our first experiment, the matrix $A \in \C^{1,000 \times 1,000}$ is diagonal with Chebyshev eigenvalues in the interval $[0.1, 200.1]$ and $\vb$ is a normalized vector of all ones and we aim to approximate $A^{-1/2}\vb$ with a relative accuracy of $10^{-6}$. In the multishift CG method, we use the optimal Zolotarev rational approximation~\cite{Zolotarev1877} for the inverse square root, which requires 15 poles for the target accuracy. As one needs to store two additional vectors per pole in the multi-shift CG method we choose a restart length of $m_{re} = 30$ in the restarted Lanczos method in order to compare methods with roughly the same memory consumption.

For a matrix with Chebyshev eigenvalues, we expect our predictions to be rather accurate as no superlinear convergence takes place. This is confirmed by Table~\ref{tab:experiment1} which shows the predicted number of matrix-vector products according to the approach outlined in Section~\ref{subsec:prediction_matvec} as well as the actual number of matrix-vector products required by our implementations. We find that the two outer-inner rational-polynomial methods are vastly outperformed by the polynomial Krylov methods. This is in particular true for the inexact extended Krylov method which needs by far the most matrix-vector products. This is not too surprising as already the solution of one linear system with $A$ requires about the same number of matrix-vector products as the multi-shift CG method. Among the polynomial methods, the multi-shift CG method needs the fewest matrix-vector products (as expected), while the restarted Lanczos method needs about 13\% fewer matrix-vector products than the two-pass Lanczos approach for this example.
\end{experiment}

\begin{experiment}\label{experiment:strakos}
The matrix in Experiment~\ref{experiment:chebyshev} is deliberately chosen so that the actual convergence  is very close to what is predicted by the worst-case bounds, informed only by the extremal eigenvalues of $A$. In reality, all eigenvalues of $A$ have an influence on the convergence  of (rational) Krylov  methods. Thus, one might get the impression that our approach for theoretically comparing the different methods holds little value in practice. While it is true that the predicted number of matrix-vector products cannot be expected to be accurate, it actually turns out that the prediction of the \emph{ratio between the numbers of matrix-vector products of the different methods} is quite accurate for very different eigenvalue distributions, in particular between the three polynomial methods. To illustrate this, we now consider a diagonal matrix with eigenvalues in $[\lmin,\lmax]$ given by
\begin{equation}\label{eq:eigenvalue_distribution}
\lambda_j = \lmin + \frac{j-1}{N-1}(\lmax-\lmin)\gamma^{N-j},\quad  j = 2,\dots, N-1
\end{equation}
with $\gamma \in (0,1)$ and where we choose $\lmin = 0.1, \lmax = 200.1$, and $N = 1000$ as before. The lower the parameter $\gamma$, the more the eigenvalues in~\eqref{eq:eigenvalue_distribution} are clustered at one end of the spectrum, making the distribution more ``favorable'' for Krylov methods, as fast spectral adaptation and hence superlinear convergence can be expected. However, as the spectral interval for~\eqref{eq:eigenvalue_distribution} stays the same, irrespective of the value of~$\gamma$, the worst-case convergence bounds we used in our predictions will be less and less sharp when $\gamma$ is decreased. 

\begin{figure}
\input{tikz/experiment2.tikz}
\caption{Left: Number of matrix-vector products required for approximating $A^{-1/2}\vb$ for a diagonal matrix $A$ with eigenvalues~\eqref{eq:eigenvalue_distribution} for different values of $\gamma$. Our predictions according to Section~\ref{subsec:prediction_matvec} are shown as dashed lines. Right: The same data as on the left-hand side, but shown as a relative number compared to the number of iterations needed by multi-shift CG.}
\label{fig:strakos_matrix}
\end{figure}
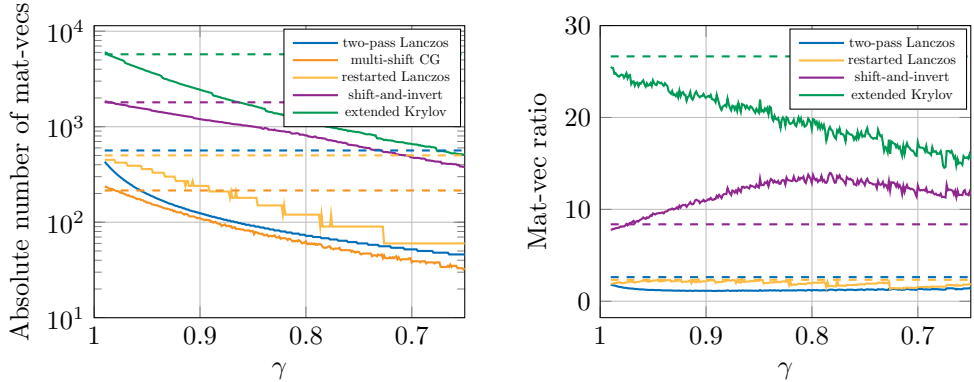
%

Figure~\ref{fig:strakos_matrix} depicts the results of applying the discussed methods to matrices with eigenvalue distributions~\eqref{eq:eigenvalue_distribution} with $\gamma\in [0.65,0.99]$. On the left-hand side, the absolute number of matrix-vector products is shown. As expected, this number decreases for decreasing $\gamma$ and thus the distance to our prediction (shown by the dashed lines) grows larger and larger. On the right-hand side, we show the \emph{relative} number of matrix-vector products of the methods, with the number needed by multi-shift CG as a baseline. For the polynomial methods, this results approximately in a horizontal line, revealing that the number of matrix-vector products these methods need in comparison to multi-shift CG stays almost constant for all the different eigenvalue distributions. Thus, although the prediction cannot be used to get a realistic estimate of the amount of work that is needed to solve a given problem, this experiment indicates that it gives a good idea of how different methods compare.
\end{experiment}

\begin{experiment}\label{experiment:work_units}
In terms of matrix-vector products, the multi-shift CG method always outperforms the other methods, which is to be expected. In this experiment, we use a very simple model of overall computational complexity to get a rough estimate of how multi-shift CG and restarted Lanczos compare in overall computation time for a given problem. To do so, we count vector operations in addition to matrix-vector products. We count the cost of one simple vector operation (addition or scaling) as one unit of work, written $1 \V$. Thus, an inner product has a cost of $\approx 2\V$. The cost $\M$ of a matrix-vector product in these units of work depends on $A$. For example, for the discretization of a differential operator on a regular two-dimensional lattice, a matrix-vector product has a cost of $\M\approx 10 \V$, while for a discretization on a three-dimensional lattice we have $\M\approx 14 \V$. We ignore the cost of all operations that are independent of $N$.

In the multi-shift CG method, the computational cost of one iteration for the ``seed system'' (i.e., the system with smallest shift) is $1\M + 12\V$ and each additional system requires an effort of $5\V$. 

For restarted Lanczos, one iteration has a cost of $1\M + 9\V$ and forming the iterate at the end of each restart cycle has a cost of $\approx 2m_{re} \V$, where $m_{re}$ is the restart length (we ignore the fact that the matrix-vector product $V_{m_{re}} \vy_{m_{re}}$ can typically be executed faster than $2{m_{re}}-1$ individual vector operations). Using these formulas, we can estimate the overall number of work units required to compute $f(A)\vb$ to a certain target accuracy by combining them with our estimates from Section~\ref{subsec:prediction_matvec}. We again use the optimal Zolotarev rational approximation for the multi-shift CG method and choose the cycle length in restarted Lanczos as $m_{re} = 2p$, where $p$ is the required number of Zolotarev poles. 

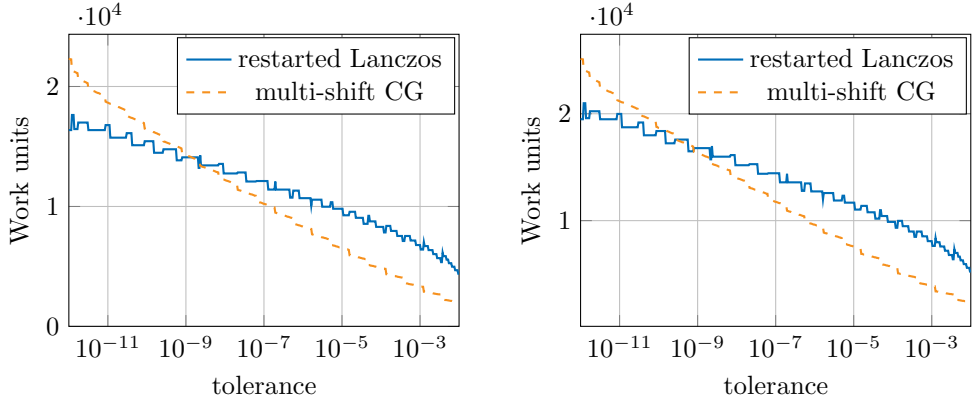
\begin{figure}
\input{tikz/experiment3.tikz}
\caption{
Estimated amount of work needed to approximate $A^{-1/2}\vb$ to the target accuracy.
Left: $2d$ discretization. Right: $3d$ discretization.}
\label{fig:estimate_work_units}
\end{figure}

To obtain a realistic comparison, we do not assume that the additional work of $5\V$ per pole is performed in all iterations for all poles, but instead use the strategy from~\cite{VanDenEshofFrommerLippertSchillingVanDerVorst2002} for removing already converged systems from the iteration in such a way that the overall error in the approximation for $f(A)\vb$ is still guaranteed to be below the target accuracy.

Figure~\ref{fig:estimate_work_units} shows the resulting estimates for the same setting as in Experiment~\ref{experiment:chebyshev} for varying target accuracies. For high target accuracies, the estimate for the restarted Lanczos method is lower than that of the multi-shift CG method, as the number of poles necessary to construct an accurate enough rational approximation increases. For the $3d$ discretization, the break-even point comes earlier, as the cost of a matrix-vector product is higher compared to a vector operation.

\end{experiment}

\section{Conclusions}\label{sec:conclusion}
Our theoretical results along with the practical comparisons in Section~\ref{sec:comparison} indicate that, among the considered outer-inner polynomial Krylov methods for approximating Stieltjes matrix functions $f(A)\vb$, both the multi-shift CG and restarted Lanczos methods are the most favourable. We find that the inexact (with polynomial inner solves) versions of the shift-and-invert method as well as the extended Krylov method are generally not competitive. 

The choice between multi-shift CG and restarted Lanczos should be informed by further considerations. In particular, the multi-shift CG method crucially requires inclusion intervals for the spectrum of $A$, and if these are not available or difficult to estimate (e.g. using a restarted Krylov method), then restarted Lanczos should be the method of choice. Note that restarted Lanczos can be implemented with deflation strategies that can further speed up the convergence. The multi-shift CG method, on the other hand, generally requires the smallest number of total matrix-vector products and might be more amendable to parallel implementation.

\section*{Acknowledgements} \review{We thank the two anonymous referees who have provided valuable comments. In particular, the approach in Remark~\ref{rem:druskin} was suggested by one of the referees.} We would like to thank Kathryn Lund for fruitful discussions and interesting ideas which led to an improvement of the manuscript, in particular the content of Section~\ref{sec:inexact_iterations}. We also thank Leonid Knizhnerman and Valeria Simoncini for some clarifying discussions.

\appendix
\section{Convergence of the shift-and-invert method with inexact solves}\label{sec:inexact_convergence}
When using inexact solves in the shift-and-invert method, we obtain the ``inexact Arnoldi decomposition'' 
$$
(B + E_m) V_m = V_{m} H_m + \beta_m \vv_{m+1} \ve_m^T.
$$
with the perturbation matrix $E_m := -B R_m V_m^*$. As $E_m$ and thus also $B+E_m$ is non-Hermitian, the results on the speed of convergence derived for Hermitian matrices are no longer applicable. We now explain why, as long as $\varepsilon := \|E_m\|$ is sufficiently small, we can still expect these results to hold in practice. Many of the techniques used in the following closely resemble the approach used in~\cite{PozzaSimoncini2019} to derive residual estimates for the inexact Arnoldi method for the matrix exponential.

%
%
%
%

We denote by $\W(M)$ the field of values (or numerical range) of a matrix $M$, by $\Delta(c,r)$ the closed disk with center $c$ and radius $r$, by $\E(c,a,b)$ the horizontal, axis-aligned ellipse with semi-axes $a > b$ and center $c$ and by $\BS(c_1,c_2,r)$ the \emph{Bunimovich stadium} with semicircle radius $r$ and semicircle centers $c_1,c_2 \in \R$; see Figure~\ref{fig:stadium_ellipse}.
As $B = (A-\xi I)^{-1}$ is Hermitian, we have 
$
\W(B) = \left[(\lmax - \xi)^{-1}, (\lmin -\xi)^{-1}\right]
$ 
and clearly $\W(E_m) \subseteq \Delta(0,\varepsilon)$. Therefore, we find 
\begin{eqnarray*}
\W(B+E_m) &\subseteq& \W(B) + \W(E_m) \\
					&=& \BS\left(\frac{1}{\lmax -\xi},\frac{1}{\lmin -\xi},\varepsilon\right)
\end{eqnarray*}

In order to derive a convergence rate for a matrix with field of values inside the Bunimovich stadium, a conformal mapping from $\C \setminus \BS\left(\frac{1}{\lmax -\xi},\frac{1}{\lmin -\xi},\varepsilon\right)$ onto $\C \setminus \Delta(0,1)$ is required. Unfortunately, no closed form for this mapping is known; see~\cite{Varma2014} for a treatment of this topic, where several numerical approximations for the conformal mapping of the Bunimovich stadium are proposed. We therefore embed $\BS\left(\frac{1}{\lmax -\xi},\frac{1}{\lmin -\xi},\varepsilon\right)$ into an ellipse as illustrated in Figure~\ref{fig:stadium_ellipse}. Specifically, we use
$$\BS\left(\frac{1}{\lmax -\xi},\frac{1}{\lmin -\xi},\varepsilon\right) \subseteq \E(a+\varepsilon,\sqrt{\varepsilon}\cdot\sqrt{2a+\varepsilon},c)$$
with $c = \frac{1}{-2\xi}$ and $a = c\cdot\beta$, where $\beta = \frac{\sqrt{\kappa}-1}{\sqrt{\kappa}+1}$.

The conformal mapping from $\C\setminus\E(a+\varepsilon,\sqrt{\varepsilon}\sqrt{2a+\varepsilon},c)$ onto $\C \setminus \Delta(0,1)$ is given by the scaled and shifted Zhukovsky map
$$\psi(z) = \frac{a}{2}\left(Rz+\frac{1}{Rz}\right)+\frac{1}{\lmax -\xi} + \frac{1}{2\xi}, \quad \text{where} \quad R = 1 + \frac{\varepsilon}{a} + \sqrt{\frac{2\varepsilon+\varepsilon^2}{\sqrt{a}}}.$$ 
Further, let $\tau > 1$ be such that $\psi(\tau) < -\frac{1}{\xi}$. Note that such $\tau$ is guaranteed to exist as long as $\varepsilon < \frac{1}{\lmax -\xi}$, which we assume from here on.

\begin{figure}
\centering
\begin{tikzpicture}

\begin{axis}[
xmin = -.5,
xmax = 6.5,
ymin = -1,
ymax = 1,
axis lines=middle,
axis equal
]

\draw[thick] (axis cs:2,-1)--(axis cs:5,-1);
\draw[thick] (axis cs:2,1)--(axis cs:5,1);
\draw[draw=black,thick] (axis cs:2,1) arc(90:270:100);
\draw[draw=black,thick] (axis cs:5,-1) arc(-90:90:100);
\draw[draw=Red,thick,dashed] (axis cs:2,0) -- (axis cs:1.2929,-0.7071);
\node[draw=none,Red] (r) at (axis cs: 1.5,-.2) {$r$};
\node[draw=none,ForestGreen] (c1) at (axis cs: 2,.3) {$c_1$};
\node[draw=none,ForestGreen] (c2) at (axis cs: 5,.3) {$c_2$};
\draw[ForestGreen,thick] (axis cs: 2,0) -- (axis cs:5,0);
\draw[NavyBlue,thick,dashed] (axis cs:3.5,0) ellipse (2.45cm and 1.9371cm);
\node[draw=none,NavyBlue] (E) at (axis cs: 5.5,-2.2) {$\E(\review{3.5},2,\review{2.5})$};
\end{axis}

\end{tikzpicture}
\caption{A Bunimovich stadium $\BS(2,\review{5},1)$ and its enclosing ellipse $\E(\review{3.5},2,\review{2.5})$.}
\label{fig:stadium_ellipse}
\end{figure}

We can then write $g(B+E_m)$ in terms of Faber polynomials $\Phi_j$ as
$$g(B+E_m) = \sum\limits_{j=0}^\infty \eta_j \Phi_j(B+E_m), \quad \text{with} \quad \eta_j = \frac{1}{2\pi i} \oint_{|z| = \tau} \frac{g(\psi(z))}{z^{j+1}}\d z.$$ 
Defining a polynomial approximation $p_{m-1}(z) = \sum\limits_{j=0}^{m-1} \eta_j\Phi_j(z)$, we have
$$\|g(B+E_m)-p(B+E_m)\| \leq \sum\limits_{j = m}^\infty \|\eta_j\Phi_j(B+E_m)\| \leq 2 \sum\limits_{j=m}^\infty |\eta_j|$$
because $\|\Phi_j(B+E_m)\| \leq 2$ due to $\W(B+E_m) \subseteq \E(a+\varepsilon,\sqrt{\varepsilon}\sqrt{2a+\varepsilon},c)$; see~\cite[Theorem 1.1]{Beckermann2005}.

Using the quasi-optimality of the Arnoldi approximation, we can conclude that 
$$\|g(B+E_m)\vb-\vg_m\| \leq 2(1+\sqrt{2})\|\vb\|\sum\limits_{j=m}^\infty |\eta_j|$$
with the Crouzeix--Palencia constant $1+\sqrt{2}$. Bounding the Faber coefficients as
$|\eta_j| \leq \frac{1}{\tau^j} \max_{|z|=\tau} |g(\psi(z))|$ 
and noting that $\max_{|z|=\tau} |g(\psi(z))| = |g(\psi(\tau))|$, we obtain the bound 
$$\|g(B+E_m)\vb-\vg_m\| \leq 4(1+\sqrt{2})\|\vb\|\frac{\tau}{\tau-1} |g(\psi(\tau))| \cdot \left|\frac{1}{\tau}\right|^m.$$
Thus, if $\|E_m\| < (\lmax -\xi)^{-1}$, so that we can find a suitable ellipse on which $g$ is analytic, we can expect convergence with a rate $\tau^{-1}$ for the perturbed problem. Similar arguments can be made concerning the decay of the Arnoldi coefficients of the perturbed iteration. It remains to investigate how $\tau^{-1}$ compares to the rate
$$\alpha(0) = \frac{\sqrt[4]{\kappa}-1}{\sqrt[4]{\kappa}+1}$$
of the unperturbed problem in dependence of $\varepsilon$. Solving $\psi(\tau) = 0$ yields
$$\tau^\ast = -\frac{c}{aR} - \sqrt{\frac{c^2}{(aR)^2}-\frac{1}{R^2}}$$
which, after straightforward algebraic manipulations gives
$$\left|\frac{1}{\tau^\ast}\right| = \frac{\beta-2\xi\varepsilon+2\sqrt{\varepsilon}\sqrt{-\beta\xi+\varepsilon\xi^2}}{1+\sqrt{1-\beta^2}}.$$
By noting that
$$\frac{\beta}{1+\sqrt{1-\beta^2}} = \frac{\sqrt[4]{\kappa}-1}{\sqrt[4]{\kappa}+1} = \alpha(0),$$
we can further rewrite this as
\begin{equation}\label{eq:perturbed_convergence_rate}
\left|\frac{1}{\tau^\ast}\right| = \alpha(0) + \frac{-2\xi\varepsilon+2\sqrt{\varepsilon}\sqrt{-\beta\xi+\varepsilon\xi^2}}{1+\sqrt{1-\beta^2}}, 
\end{equation}
which more clearly reveals how the convergence rate deteriorates with growing $\varepsilon$.
For $\varepsilon \ll 1$, the term involving $\sqrt{\varepsilon}$ dominates the perturbation given in~\eqref{eq:perturbed_convergence_rate}. Thus, we can expect the convergence rate to deteriorate approximately like $\sqrt{\varepsilon}$.
%

\begin{figure}
\input{tikz/convrate.tikz}
\caption{Left: Convergence rate for the perturbed problem as a function of $\varepsilon = \|E_m\|$. The horizontal, dashed line shows the convergence rate from Theorem~\ref{the:shift_invert} for the unperturbed problem and the vertical dash-dotted line shows $(\lmax -\xi)^{-1}$. Right: Convergence \review{of} the exact and inexact shift-and-invert method \review{for the same matrix and vector as in Experiment~\ref{experiment:chebyshev}} when aiming for an overall tolerance of $\texttt{tol} = 10^{-6}$ and using the relaxation strategy outlined in Section~\ref{sec:inexact_iterations}.}
\label{fig:perturbed_convergence_rate}
\end{figure}
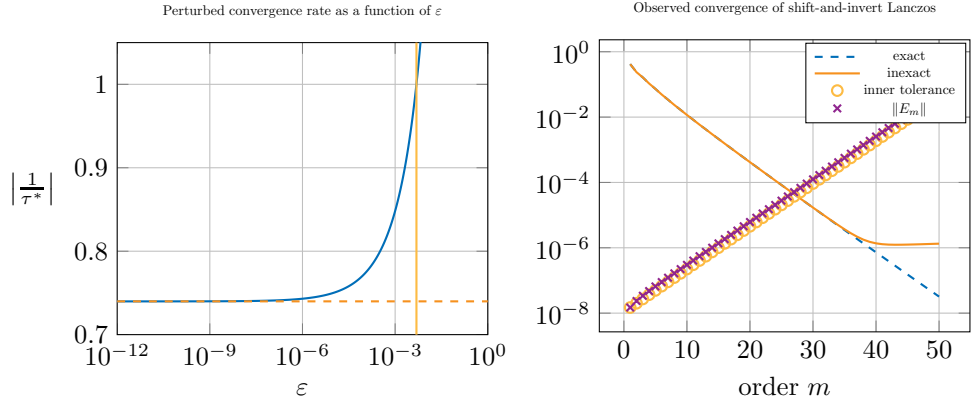

In Figure~\ref{fig:perturbed_convergence_rate}, we illustrate how $|\tau^\ast|^{-1}$ evolves with growing $\varepsilon$ for a diagonal matrix $A$ with $N=1000$ Chebyshev eigenvalues in the interval $[0.1,200.1]$. For this example, the smallest eigenvalue of $B$ is $(\lmax -\xi)^{-1} \approx 0.0049$. As predicted by our arguments above, for small $\varepsilon$, the convergence rate is essentially the same as that of the unperturbed problem. With growing $\varepsilon$, the convergence rate deteriorates until it reaches the value $1$ for $\varepsilon =(\lmax -\xi)^{-1}$. In that case, we cannot guarantee convergence as any longer $-\frac{1}{\xi} \in \E(a+\varepsilon,\sqrt{\varepsilon}\cdot\sqrt{2a+\varepsilon},c)$, which is a singularity of $g$.

Finally, we address  how the tolerance of the inner iteration affects $\varepsilon = \|E_m\|$. Fortunately, this is rather easy. From the definition of $E_m$, we have
$$\|E_m\| = \|-BR_mV_m^T\| = \|BR_m\| \leq  \|B\| \sqrt{\sum\limits_{j=0}^m \|\vr_j\|^2},$$
so that we can control it via the residual norms of the inner iterations.


In Figure~\ref{fig:perturbed_convergence_rate} (right), we apply the shift-and-invert method to \review{the diagonal example matrix already used in Experiment~\ref{experiment:chebyshev}} and also show the norm of the error matrix $\|E_m\|$ in each step.
Comparing the norm $\varepsilon$ of the error matrix in Figure~\ref{fig:perturbed_convergence_rate} (right) with the convergence rate given for these values of in $\varepsilon$ in Figure~\ref{fig:perturbed_convergence_rate} (left), we would expect the convergence rate of the inexact method to deteriorate much earlier than it does in reality. 

There are different factors playing a role in the explanation of this effect: our results are valid for matrices with field of values in an ellipse which encloses the Bunimovich stadium, i.e., we have chosen a set containing the field of values of $B+E_m$ which is larger than necessary, as otherwise we have not been able to construct a conformal mapping. 
In addition, the convergence rate is only a worst-case estimate. While in the unperturbed case we know that convergence for a matrix with Chebyshev eigenvalues will closely follow this worst-case bound, it is not clear how closely the perturbed matrix follows the worst-case bound for the perturbed case. Finally, we used the trivial inclusion
$\W(E_m) \subseteq \Delta(0,\|E_m\|),$ 
which often overestimates the actual diameter of $\W(E_m)$.

\bibliographystyle{siam}
\bibliography{matrixfunctions}
\end{document}

%% file: tikz/experiment2.tikz.tex

\pgfplotsset{height=0.42\linewidth,width=0.5\linewidth,compat=1.10,every axis/.append style={legend style={/tikz/every even column/.append style={column sep=6pt}}}}

\noindent%
\begin{tikzpicture}[scale=1]%
    \begin{semilogyaxis}[legend style={at={(.99,.99)},nodes={scale=0.53, transform shape}}, 
   	anchor= north east, legend columns=1,
   	xmin=0.65, xmax=1,grid=major, ymin = 1e1,
   	xlabel={$\gamma$}, ylabel={Absolute number of mat-vecs},x dir=reverse]

\addplot[color=NavyBlue,thick]
table [x ={gamma},y ={2pl}] {tikz/experiment2.dat};\addlegendentry{two-pass Lanczos}
 
  \addplot[color=BurntOrange,thick]
table [x ={gamma},y ={mscg}]{tikz/experiment2.dat} node [pos=0,left] {}; \addlegendentry{multi-shift CG}

\addplot[color=Dandelion,thick]
table [x ={gamma},y ={rlan}]{tikz/experiment2.dat} node [pos=0,left] {}; \addlegendentry{restarted Lanczos}

\addplot[color=Plum,thick]
table [x ={gamma},y ={si}]{tikz/experiment2.dat} node [pos=0,left] {}; \addlegendentry{shift-and-invert}

\addplot[color=ForestGreen,thick]
table [x ={gamma},y ={ext}]{tikz/experiment2.dat} node [pos=0,left] {}; \addlegendentry{extended Krylov}

\addplot[color=NavyBlue,thick,dashed]
table [x ={gamma},y ={pred2pl}] {tikz/experiment2.dat};
 
  \addplot[color=BurntOrange,thick,dashed]
table [x ={gamma},y ={predmscg}]{tikz/experiment2.dat} node [pos=0,left] {}; 
\addplot[color=Dandelion,thick,dashed]
table [x ={gamma},y ={predrlan}]{tikz/experiment2.dat} node [pos=0,left] {}; 

\addplot[color=Plum,thick,dashed]
table [x ={gamma},y ={predsi}]{tikz/experiment2.dat} node [pos=0,left] {}; 
\addplot[color=ForestGreen,thick,dashed]
table [x ={gamma},y ={predext}]{tikz/experiment2.dat} node [pos=0,left] {}; 

    \end{semilogyaxis}
\end{tikzpicture}%
\hfill%
\begin{tikzpicture}[scale=1]%
    \begin{axis}[legend style={at={(.99,.99)},nodes={scale=0.53, transform shape}}, 
   	anchor= north east, legend columns=1,
   	xmin=0.65, xmax=1,grid=major, ymax=30,
   	xlabel={$\gamma$}, ylabel={Mat-vec ratio},x dir=reverse]

\addplot[color=NavyBlue,thick]
table [x ={gamma},y ={2pl}] {tikz/experiment2_relative.dat};\addlegendentry{two-pass Lanczos}

\addplot[color=Dandelion,thick]
table [x ={gamma},y ={rlan}]{tikz/experiment2_relative.dat} node [pos=0,left] {}; \addlegendentry{restarted Lanczos}

\addplot[color=Plum,thick]
table [x ={gamma},y ={si}]{tikz/experiment2_relative.dat} node [pos=0,left] {}; \addlegendentry{shift-and-invert}

\addplot[color=ForestGreen,thick]
table [x ={gamma},y ={ext}]{tikz/experiment2_relative.dat} node [pos=0,left] {}; \addlegendentry{extended Krylov}

\addplot[color=NavyBlue,thick,dashed]
table [x ={gamma},y ={pred2pl}] {tikz/experiment2_relative.dat};
 
\addplot[color=Dandelion,thick,dashed]
table [x ={gamma},y ={predrlan}]{tikz/experiment2_relative.dat} node [pos=0,left] {}; 

\addplot[color=Plum,thick,dashed]
table [x ={gamma},y ={predsi}]{tikz/experiment2_relative.dat} node [pos=0,left] {}; 
\addplot[color=ForestGreen,thick,dashed]
table [x ={gamma},y ={predext}]{tikz/experiment2_relative.dat} node [pos=0,left] {}; 

    \end{axis}
\end{tikzpicture}

%% file: tikz/experiment3.tikz.tex

\pgfplotsset{height=0.42\linewidth,width=0.52\linewidth,compat=1.10,every axis/.append style={legend style={/tikz/every even column/.append style={column sep=6pt}}}}

\noindent%
\begin{tikzpicture}[scale=1]%
    \begin{semilogxaxis}[legend style={at={(.99,.99)}}, 
   	anchor= north east, legend columns=1,
   	xmin=1e-12, xmax=1e-2,grid=major, 
   	xlabel={tolerance}, ylabel={Work units}]

\addplot[color=NavyBlue,thick]
table [x ={tol},y ={wrlan}] {tikz/experiment3_2d.dat};\addlegendentry{restarted Lanczos}
 
  \addplot[color=BurntOrange,thick,dashed]
table [x ={tol},y ={wcg}]{tikz/experiment3_2d.dat} node [pos=0,left] {}; \addlegendentry{multi-shift CG}

    \end{semilogxaxis}
\end{tikzpicture}
\hfill
\begin{tikzpicture}[scale=1]%
    \begin{semilogxaxis}[legend style={at={(.99,.99)}}, 
   	anchor= north east, legend columns=1,
   	xmin=1e-12, xmax=1e-2,grid=major, 
   	xlabel={tolerance}, ylabel={Work units}]

\addplot[color=NavyBlue,thick]
table [x ={tol},y ={wrlan}] {tikz/experiment3_3d.dat};\addlegendentry{restarted Lanczos}
 
  \addplot[color=BurntOrange,thick,dashed]
table [x ={tol},y ={wcg}]{tikz/experiment3_3d.dat} node [pos=0,left] {}; \addlegendentry{multi-shift CG}

    \end{semilogxaxis}
\end{tikzpicture}

%% file: tikz/convrate.tikz.tex

\pgfplotsset{height=0.42\linewidth,width=0.5\linewidth,compat=1.10,every axis/.append style={legend style={/tikz/every even column/.append style={column sep=6pt}}}}

\noindent%
\begin{tikzpicture}[scale=1]%
    \begin{semilogxaxis}[legend style={at={(.99,.99)}}, 
   	anchor= north east, legend columns=1,ylabel style={rotate=-90},
   	xmin=1e-12, xmax=1, grid=major, ymin=0.7, ymax=1.05,
   	xlabel={$\varepsilon$}, ylabel={$\left|\frac{1}{\tau^\ast}\right|$}, title={\scalebox{.53}{Perturbed convergence rate as a function of $\varepsilon$}}]

\addplot[color=NavyBlue,thick]
table [x ={tol},y ={rate}] {tikz/convrate.dat};
 
\draw[thick, dashed, color=BurntOrange] (axis cs:1e-12,.7399) -- (axis cs:1,.7399);

\draw[thick, color=Dandelion] (axis cs:.0049,.7) -- (axis cs:.0049,1.05);

    \end{semilogxaxis}
\end{tikzpicture}
\hfill
\begin{tikzpicture}[scale=1]%
    \begin{semilogyaxis}[legend style={at={(.99,.99)},nodes={scale=0.53, transform shape}}, 
   	anchor= north east, legend columns=1,
   	grid=major, xlabel={order $m$}, title={\scalebox{.53}{Observed convergence of shift-and-invert Lanczos}}]

\addplot[color=NavyBlue,thick,dashed]
table [x ={x},y ={err}] {tikz/inexact.dat};\addlegendentry{exact}
 
  \addplot[color=BurntOrange,thick]
table [x ={x},y ={err_inex}] {tikz/inexact.dat};\addlegendentry{inexact}

\addplot[color=Dandelion,thick,mark=o,only marks]
table [x ={x},y ={tol}] {tikz/inexact.dat};\addlegendentry{inner tolerance}

\addplot[color=Plum,thick,mark=x,only marks]
table [x ={x},y ={normE}] {tikz/inexact.dat};\addlegendentry{$\|E_m\|$}

    \end{semilogyaxis}
\end{tikzpicture}